\documentclass[11pt]{amsart}
\usepackage{centernot}
\usepackage{mathtools}
\usepackage{stmaryrd}
\usepackage{tikz-cd}
\usepackage{enumitem}

\makeatletter
\newcommand{\xMapsto}[2][]{\ext@arrow 0599{\Mapstofill@}{#1}{#2}}
\def\Mapstofill@{\arrowfill@{\Mapstochar\Relbar}\Relbar\Rightarrow}
\makeatother

\usepackage[mathcal]{euscript}
\usepackage{amssymb, amsfonts, amsmath}
\usepackage{xypic}
\usepackage{amscd}
\usepackage{dutchcal}
\usepackage{amsthm}
\usepackage[margin=1.5in]{geometry}
\usepackage{xcolor}
\usepackage{comment}

\usepackage[style = alphabetic, doi=false]{biblatex}
\usepackage{hyperref}
\usepackage{xurl}
\hypersetup{breaklinks=true}
\renewbibmacro{in:}{}
\definecolor{winered}{rgb}{0.5,0,0}
 \usepackage{hyperref}
 \hypersetup{citecolor = winered, linkcolor = winered, menucolor = winered, colorlinks = true}

\newtheorem{theorem}{Theorem}[section]
\newtheorem{lemma}[theorem]{Lemma}
\newtheorem{corollary}[theorem]{Corollary}

\newtheorem{theorem-definition}[theorem]{Theorem-Definition}
\newtheorem{proposition}[theorem]{Proposition}

\theoremstyle{definition}
\newtheorem{definition}[theorem]{Definition}

\newtheorem{remark}[theorem]{Remark}
\newtheorem{warning}[theorem]{Warning}
\newtheorem*{theorem*}{Theorem}

\numberwithin{equation}{section} \numberwithin{figure}{section}

\numberwithin{equation}{section}

\newcommand{\Z}{\mathbb{Z}}

\newcommand{\E}{\E_{\infty}}

\newcommand{\la}{\lambda}

\newcommand{\cU}{\mathcal{U}}

\newcommand{\HDco}{\mathsf{D}^{\mathsf{co}}}
\newcommand{\Dco}{\mathcal{D}^{\mathsf{co}}}
\newcommand{\De}{\mathcal{D}}
\newcommand{\HDe}{\mathsf{D}}

\newcommand{\HH}{\operatorname{HH}}
\newcommand{\HHcx}{\operatorname{\mathcal{CH}}}
\newcommand{\HN}{\operatorname{HN}}
\newcommand{\HNcx}{\operatorname{\mathcal{CN}}}

\newcommand{\HCcx}{\operatorname{\mathcal{CC}}}
\newcommand{\coHH}{\operatorname{coHH}}
\newcommand{\coHHcx}{\operatorname{co\mathcal{CH}}}
\newcommand{\coHN}{\operatorname{coHN}}
\newcommand{\coHNcx}{\operatorname{co\mathcal{CN}}}
\newcommand{\coHC}{\operatorname{coHC}}
\newcommand{\coHCcx}{\operatorname{co\mathcal{CC}}}
\newcommand{\hocolim}{\operatorname{hocolim}}
\newcommand{\Gr}{\operatorname{Gr}}
\newcommand{\im}{\operatorname{Im}}
\newcommand{\cone}{\operatorname{cone}}
\newcommand{\Ba}{\mathrm{B}}
\newcommand{\cotensorte}{\boxempty^{\tau^e}_{C^e_0}}
\DeclareMathOperator{\Ob}{Ob}
\DeclareMathOperator{\Mod}{\!-Mod}
\DeclareMathOperator{\Comod}{\!-Comod}

\newcommand{\dgCat}{\mathsf{dgCat}} 
\newcommand{\ptdco}{\mathsf{cuCoa}^{\mathsf{ptd}}_*}
\newcommand{\Ho}{\operatorname{Ho}}
\newcommand{\kob}[1]{{k[\Ob(#1)]}}

\usepackage{scalerel,amssymb}

\newcommand{\id}{\mathbf{1}}
\newcommand{\op}{^{\mathsf{op}}}
\DeclareMathOperator{\Hom}{Hom}
\DeclareMathOperator{\RHom}{R\mathcal{Hom}}
\DeclareMathOperator{\uHom}{\mathcal{Hom}}

\usepackage{graphicx}

\begin{document}

\title[]{Koszul duality and Calabi-Yau structures}
\author{Julian Holstein}
\address{Julian Holstein\\Fachbereich Mathematik\\Universität Hamburg\\
Bundesstraße 55, 20146 Hamburg, Germany}
	\email{julian.holstein@uni-hamburg.de}

\author{Manuel Rivera}
\address{
Manuel Rivera\\
Purdue University\\Department of Mathematics\\ 150 N. University St. West Lafayette, IN 47907, USA
}
\email{manuelr@purdue.edu}

\begin{abstract}
    We show that Koszul duality between differential graded categories and pointed curved coalgebras interchanges smooth and proper Calabi-Yau structures. 
    This result is a generalization and conceptual explanation of the following 
    two applications. For a finite-dimensional Lie algebra 
    a smooth Calabi-Yau structure on the universal enveloping algebra is equivalent to 
    a proper Calabi-Yau structure on the Chevalley-Eilenberg chain coalgebra, which 
    exists if and only if  Poincaré duality is satisfied. 
    For a topological space $X$ having the homotopy type of a finite complex we show an oriented Poincaré duality structure (with local coefficients) on $X$ is equivalent to a proper Calabi-Yau structure on the dg coalgebra of chains on $X$ and to a smooth Calabi-Yau structure on the dg algebra of chains on the based loop space of $X$. 
\end{abstract}
\maketitle

\tableofcontents
\section{Introduction}

One of the many versions of Koszul duality provides a bar-cobar adjunction
\begin{eqnarray*}
    \Ba \colon \mathsf{dgAlg}_k^{\text{aug}} \rightleftarrows \mathsf{dgCoalg}_k^{\text{conil}} \colon \Omega
\end{eqnarray*} 
between the categories of augmented dg algebras and conilpotent dg coalgebras over a field $k$ that defines an equivalence of suitable homotopy theories \cite{lefèvrehasegawa2003surlesainfinicategories, positselski2011two, positselski2023differential}. In this article, we prove that, under this duality, smooth (proper) Calabi-Yau structures on dg algebras correspond to proper (smooth) Calabi-Yau structures on dg coalgebras. In fact, we prove this result for a many-object version of Koszul duality that relates dg categories and pointed curved coalgebras \cite{holstein2022categorical}. 
We also consider the case that $k$ is a principal ideal domain whenever possible. For ease of exposition, in this introduction we explain the statement and its significance in more detail in the case of conilpotent dg coalgebras and augmented dg algebras over a field.

A dg $k$-algebra $A$ is said to be smooth if it is compact in the derived category of dg $A$-bimodules \cite{kontsevich2008notes}.
The notion of weak equivalence in the context of dg $A$-bimodules is quasi-isomorphism. A dg algebra $A$ is smooth if and only if it is quasi-isomorphic  to a retract of a finite iterated extension of shifts of $A^e=A \otimes A^{\text{op}}$. 
A smooth $n$-Calabi-Yau structure on a smooth dg algebra $A$ is a cycle $\alpha \in \HNcx_n(A)$ in the negative cyclic complex of $A$ inducing an equivalence of $A$-bimodules $ \phi(\alpha) \colon A^! \xrightarrow{\simeq} A[n]$ in the derived category, where $A^!$ is defined as the derived dual $\RHom_{A^e}(A, A^e)$ \cite{ginzburg2007calabiyaualgebras, brav2019relative}.
On the other hand, $A$ is proper if it is perfect (or, equivalently, compact) in the derived category of $k$-modules. Thus, smooth and proper are two different finiteness conditions on a dg algebra. A proper $n$-Calabi-Yau structure on a proper dg algebra $A$ is a dual cyclic chain $\HCcx_*(A) \to k[n]$ inducing an equivalence of $A$-bimodules $A[-n] \xrightarrow{\simeq} A^*$ in the derived category, where $A^*$ is the $k$-linear dual of $A$. 

In this article, we consider a notion of properness for coalgebras. We define a conilpotent dg $k$-coalgebra $C$ to be \textit{proper} if it is compact in the ``coderived category" of dg $C$-bicomodules. The coderived category is defined not through quasi-isomorphisms of dg $C$-bicomodules but rather by a stronger notion of weak equivalence: maps of dg $C$-bicomodules that become quasi-isomorphisms of $\Omega C$-bimodules after applying a version of the cobar functor. This cobar functor may be understood as a certain twisted tensor product with $\Omega C$. A dg coalgebra $C$ is proper if and only if it is weakly equivalent to a retract of a $C$-bicomodule whose underlying graded $k$-module is finite dimensional. 
A proper dg coalgebra $C$ has a ``dual" object $C^\vee$ when considered as a dg $C$-bicomodule.
If $C$ is finite-dimensional, this is just the linear dual $\uHom_k(C,k)$ with its natural dg $C$-bicomodule structure.

A \emph{proper $n$-Calabi-Yau structure} on a proper dg coalgebra $C$ is a cycle $\beta \in \coHNcx_n(C)$ in the negative cocyclic complex of $C$ inducing an equivalence of $C$-bicomodules $\varphi(\beta) \colon C^\vee \xrightarrow{\simeq} C[n]$ in the coderived category.
We also define the notion of weakly smooth dg coalgebras and  smooth Calabi-Yau structures on them, corresponding to a proper Calabi-Yau structure on a proper dg algebra. 
In the context of augmented dg algebras and conilpotent dg coalgebras, our main theorem says the following.

\begin{theorem}\label{thm:mainintro}
A proper $n$-Calabi-Yau structure on a proper conilpotent dg coalgebra $C$ is equivalent to a smooth $n$-Calabi-Yau structure on the augmented dg algebra $\Omega C$, where $\Omega$ denotes the cobar functor.

A proper $n$-Calabi-Yau structure on a proper augmented dg algebra $A$ is equivalent to a smooth $n$-Calabi-Yau structure on the conilpotent dg coalgebra $BA$, where $B$ denotes the bar functor. 
\end{theorem}
These are special cases of Theorems \ref{thm:main} and \ref{thm:dualmain}, which give the corresponding statement for Calabi-Yau structures on dg categories and pointed curved coalgebras (recalled in Section \ref{sect:ptdco}). The proof of Theorem \ref{thm:mainintro} relies on the equivalence of the coderived category of dg $C$-bicomodules and the derived category of dg $\Omega C$-bimodules and a natural isomorphism between negative cocyclic homology of $C$ and negative cyclic homology of $\Omega C$. 

Our results may be understood as a natural generalization of the following two phenomena, which now follow as applications of Theorem \ref{thm:mainintro}. 

First note that the notion of proper Calabi-Yau structure, which, in the finite dimensional case, establishes a shifted duality between linear duals, is reminiscent of Poincaré duality of closed oriented manifolds. 
In the context of coalgebras, where the notion of equivalence is stronger than quasi-isomorphism, we have to strengthen the duality condition and the existence of a proper Calabi-Yau structure is equivalent to a map to the shifted linear dual that remains an equivalence when taking coefficients in bicomodules.
If our coalgebra is cocommutative, or if it 
has an $E_2$-coalgebra structure and the resulting cobar dg bialgebra has the property of being a Hopf algebra, this simplifies to 
coefficients in (one-sided) comodules. 
This is satisfied in the topological setting of singular chains and the proper Calabi-Yau structure is equivalent to a Poincar\'e duality structure with local coefficients. Thus our first application is the following:

 \begin{theorem}\label{thm:poincareintro} Let $(X,b)$ be a pointed path-connected topological space having the homotopy type of a finite simplicial complex and $\alpha_X \in C_n(X;k)$ an $n$-cycle in the complex of singular chains on $X$. The following are equivalent:
 \begin{enumerate}
 \item (Poincaré duality structure) Taking the cap product with $\alpha_X$ induces a degree $n$ quasi-isomorphism
\[ C^{*}(X;\ell) \to C_{n-*}(X; \ell),\]
from cochains to chains with coefficients in any local system (left $k[\pi_1(X,b)]$-module) $\ell$. 
 \item (Proper Calabi-Yau structure) The cycle $\alpha_X$ induces a proper Calabi-Yau structure on the conilpotent dg coalgebra of pointed normalized singular chains $C_*(X,b;k)$.
 \item (Smooth Calabi-Yau structure) The cycle $\alpha_X$ induces a smooth Calabi-Yau structure on the augmented dg algebra $C_*(\Omega_bX;k)$ of singular chains on the space $\Omega_bX$ of Moore loops in $X$ based at $b$.
 \end{enumerate}
 \end{theorem}
 This result places previous results \cite{brav2019relative, cohencalabi} about Calabi-Yau structures and Poincar\'e-duality into the context of Koszul duality. It uses the existence of a natural quasi-isomorphism of dg algebras between $\Omega C_*(X,b;k)$ and the singular chains on $\Omega_bX$, a result extending classical theorem of Adams' to the non-simply connected setting \cite{rivera-zeinalian, riveracobar}. In particular, using that any representative of the fundamental class of an oriented closed manifold $X$ satisfies (1), (2) may be interpreted as a chain level lift of classical Poincaré duality to a homotopy coherent (or $A_{\infty}$-) map of $C_*(X,b;k)$-bicomodules, which is a weak equivalence (a notion which is strictly stronger than quasi-isomorphism in the non-simply connected context). We give meaning to this statement even when $k$ is a PID. Furthermore, (2) packages the Poincaré duality structure in terms of minimal chain-level data.  This is potentially useful when studying and computing the string topology of non-simply connected manifolds, where one uses explicit models for chain level Poincaré duality to construct and study intersection type operations on spaces of loops and paths or Hochschild complexes.

For a second example we consider a finite-dimensional Lie algebra $\mathfrak g$ satisfying
Poincar\'e duality, which holds if and only if $\mathfrak g$ is unimodular.
We may translate the Poincar\'e duality map to a proper Calabi-Yau structure on the Chevalley-Eilenberg chain coalgebra, which is equivalent to a smooth Calabi-Yau structure on the Koszul dual algebra.
\begin{theorem}[Theorem \ref{thm:liekoszul}]
    Let $\mathfrak g$ be an $n$-dimensional Lie algebra. 
    Then the Chevalley-Eilenberg chain coalgebra $C_*(\mathfrak g)$ has a proper $n$-Calabi-Yau structure if and only if the universal enveloping algebra $\mathcal U(\mathfrak g)$ has a smooth $n$-Calabi-Yau structure.
    
    This is the case if and only if $\mathfrak g$ is unimodular.
\end{theorem}

\begin{remark}
    We should explain our nomenclature. 
   The definition of a \emph{proper} dg coalgebra is formally similar to that of a smooth (rather than proper) dg algebra.
   However, the derived category of modules and the coderived category of comodules behave quite differently. 
   Thus we call the natural Calabi-Yau structure on such a coalgebra $C$ \emph{proper} as it identifies $C$ with (a version of) its linear dual, see Definition \ref{def:cyproper}, as is the case for the proper Calabi-Yau structure on a dg algebra (or category). 

   Moreover, coalgebras are often best thought of as (contravariantly equivalent to) pseudocompact algebras and this equivalence identifies those coalgebras Koszul dual to proper algebras with a generalization of the notion of \emph{smooth} pseudocompact algebras, see Remark \ref{rmk:weaklysmooth}. Thus we call them weakly smooth coalgebras and their natural Calabi-Yau structures smooth.

   One word of warning: Smooth Calabi-Yau structures on dg algebras (or dg categories) are also called \emph{left} Calabi-Yau structures as tensoring with the bimodule has a left adjoint, and dually proper Calabi-Yau structures on dg categories are called \emph{right} Calabi-Yau, see \cite{brav2019relative}. 
   We cannot do this for coalgebras, as a proper Calabi-Yau structure on a coalgebra corresponds to the existence of a left adjoint and a smooth Calabi-Yau structure to the existence of a right adjoint.
\end{remark}

\subsection{Related work}
While completing this article, we learned that Matt Booth, Joe Chuang and Andrey Lazarev are also studying the same question and have obtained similar results which have since become available at \cite{booth2025nonsmoothcalabiyaustructuresalgebras}.

The Koszul duality between different kinds of Calabi-Yau structure has appeared in draft form in \cite[Theorem 25]{cohencalabi} as the statement that given two Koszul dual dg algebras $A$ and $B$, i.e.\ $A \simeq R\uHom_B(k,k)$ and $B \simeq R\uHom_A(k,k)$, $A$ has a smooth Calabi-Yau structure if and only if $B$ has a compact Calabi-Yau structure in the sense of \cite{kontsevich2008notes}.
They put this in the context of the Koszul dual Calabi-Yau algebras giving dual field theories, an idea going back to Kontsevich. 
Note that our version of Koszul duality is more general. In particular, when considering the cochain algebra of a topological space, the result of \cite{cohencalabi} would only apply to simply connected spaces, since in order for the algebra of cochains on a space $X$ and the algebra of chains on the based loop space $\Omega X$ to be Koszul dual in their sense one needs restrictions on the fundamental group. However, the \textit{coalgebra} of singular chains on $X$ and the algebra of chains on the based loop space $\Omega X$ are Koszul dual in our sense without any restrictions on the fundamental group. 

The interpretation of Poincaré duality given by (2) in Theorem \ref{thm:poincareintro} may be understood as a strengthening of the structure described in \cite{tradler-zeinalian} by allowing arbitrary coefficients and incorporating the action of the fundamental group (and local coefficients) at the level of chains made possible by working under a notion of weak equivalence stronger than quasi-isomorphism.

Previous proofs that chains on the loop space of an oriented manifold have a smooth Calabi-Yau structure also relate it to Poincar\'e duality with local coefficients, see \cite{cohencalabi} and
\cite[Section 5.1]{brav2019relative}, as well as \cite{jorgensen2008calabi} in the simply connected case. Similar results have also been obtained in \cite{shendetakeda}. Our proof is explicit and self-contained using the bar-cobar formalism. 

Another related result is \cite[Theorem 6.6]{herscovich2023cyclic}, which builds on previous work of Van den Bergh \cite{van2015calabi} and shows that 
a local pseudocompact dg algebra has a smooth $d$-Calabi-Yau structure if and only if the cobar construction of its linear dual has a strong homotopy inner product of degree $d$, which amounts to a compact $d$-Calabi-Yau structure \cite[Theorem 5.3, Corollary 5.4]{herscovich2023cyclic}.
As local pseudocompact algebras are contravariantly equivalent to conilpotent coalgebras, this is a one-object version of our Theorem \ref{thm:dualmain} in the special case of a coalgebra that is smooth (see Remark \ref{rmk:weaklysmooth} for the distinction between smooth and weakly smooth coalgebras).

Poincar\'e duality and the existence of a weak smooth Calabi-Yau structure on the universal enveloping algebra for unimodular Lie algebras have been shown independently in \cite{hazewinkel1970duality} and \cite{he2010cocommutative}. 

\subsection{Outline}
We begin by recalling some necessary background in Section \ref{sect:background}.
In Section \ref{sect:ptdco} we recall the category of pointed curved coalgebras which is Koszul dual to the category of dg categories.
In Section \ref{sect:cat} we exhibit dg categories as monoids in bicomodules, a useful technical trick that makes Koszul duality convenient to state in Section \ref{sect:koszul}.
We recall the basics of Calabi-Yau structures on algebras and categories in Section \ref{sect:cy}.

We proceed to analyse the coderived category of comodules more closely in Section \ref{sect:comodules} where we describe maps in the coderived category as homotopy comodule maps in Corollary \ref{cor:dcohoms}.
This analysis applies in particular without the assumption that $k$ is a field, when the model structure on comodules is not available.
Then we can define proper pointed curved coalgebras and study their dual comodules in Section \ref{sect:coproper}.

In Section \ref{sect:propercy} we define proper Calabi-Yau structures on coalgebras and show
that these exactly correspond to smooth Calabi-Yau structures on the Koszul dual categories, see Theorem \ref{thm:main}.

The dual situation is considered in \ref{sect:smoothcy}, where we define smooth pointed curved coalgebras, and their smooth Calabi-Yau structures and show in Theorem \ref{thm:dualmain} that these correspond to proper Calabi-Yau structures on locally proper dg categories.

In Section \ref{sect:lie} we analyse proper Calabi-Yau structures on cocommutative pointed curved coalgebras and consider the example of the Chevalley-Eilenberg chain coalgebras of unimodular Lie algebras. 

In Section \ref{sect:manifold} we conclude by a detailed analysis of the Calabi-Yau structures associated to topological spaces.
\subsection{Acknowledgements}
The authors are grateful for fruitful conversations on the topics of this article with
Matt Booth, Joe Chuang, Tobias Dyckerhoff, Bernhard Keller, Andrey Lazarev, Alex Takeda, Thomas Tradler, and Mahmoud Zeinalian.

	The first author acknowledges support by the Deutsche Forschungsgemeinschaft (DFG, German Research Foundation) through EXC 2121 ``Quantum Universe'' -- project number 390833306 -- and SFB 1624 ``Higher structures, moduli spaces and integrability'' -- project number 506632645.  The second author acknowledges support by NSF Grant DMS 2405405 and by Shota Rustaveli National Science Foundation of Georgia (SRNSFG) [grant number FR-23-5538].
 
\section{Background}\label{sect:background}
\subsection{Definitions and notation}\label{sect:defs}
We fix a principal ideal domain $k$ as a ground ring, with particular interest in the case $k = \Z$. 
We will often consider the case when $k$ is a field, this will always be indicated explicitly.
As a rule, the theory becomes more convenient and elegant when $k$ is a field, but the key statements continue to hold when $k$ is a general PID.
Undecorated $\Hom$, $\uHom$, $\RHom$ and $\otimes$ are considered over $k$. Ordinary hom spaces will be denoted by $\Hom$ while hom complexes are denoted by $\uHom$ and derived hom complexes by $\RHom$. Derived functors are generally denoted by $L$ and $R$. The symbol $\simeq$ means ``weakly equivalent" and $\cong$ means ``isomorphic".

We use \emph{homological} grading conventions, that is, differentials are of degree $-1$. This puts our twisting cochains into degree $-1$ and differs from the convention in \cite{brav2019relative} but Hochschild homology and the objects in our key examples are naturally homologically graded. We use the Koszul sign convention throughout and do not always specify the signs.

By $\infty$-category we shall mean a quasi-category (with a potentially large set of vertices).

\subsection{Curved coalgebras}\label{sect:ptdco}
Let $C=(C, \Delta, \epsilon, d, h)$ be a pointed curved coalgebra over $k$, also called a curved pointed coalgebra. Recall that this means 
\begin{enumerate}
    \item $(C, \Delta)$ is a graded coassociative coalgebra over $k$.
    If $k$ is not a field we assume that $C$ as a graded $k$-module is free. 
    We will still call $C$ a coalgebra rather than a coring.
      \item $d \colon C \to C$ is a linear endomorphism of degree $-1$, which is a coderivation of the coproduct $\Delta$.
    \item $C$ has a maximal cosemisimple subcoalgebra $C_0$ that is isomorphic to a sum of copies of the ground ring (viewed as trivial coalgebras) and  $\epsilon: C \to C_0$ is a splitting of the natural inclusion compatible with $d$. Furthermore, we require $d$ to be zero on $C_0$. We denote the kernel of $\epsilon$ by $\overline{C}$.
\item $h$ is a curvature for $C$, i.e.\ a homogeneous linear function $h : C \to k$ of degree $-2$ such that $d^2(c) = \{h,c\}$ where the right hand side denotes the commutator with the natural two-sided action of $C^* = \uHom(C,k) \cong \uHom_{C_0}(C,C_0)$ on $C$.

\end{enumerate}
The pointed curved coalgebras together with a final object form the category $\ptdco$. (The morphisms are defined in \cite{holstein2022categorical}.)
\begin{warning}
$C_0$ for us always denotes the maximal cosemisimple, not the degree $0$ part of $C$.
\end{warning}

We will sometimes consider the case that $C$ has finite rank, i.e.\ $C$ is concentrated in finitely many degrees and has finite rank in each degree.

A left $C$-comodule 
 is a  left comodule $M$ over the underlying graded coalgebra of $C$, equipped with a linear endomorphism $d_M$ of degree $-1$ that is compatible with differential and curvature on $C$.
 The splitting of $C_0 \to C$ equips any left $C$-comodule $M$ with the structure of a left $C_0$-comodule.
 We denote by $C\Comod$ the category of left comodules over $C$ and by $C\op\Comod$ the analogously defined category of right comodules over $C$. 
 
We define the ``enveloping coalgebra'' $C^e = C \otimes C\op$, which is a curved pointed coalgebra with $C^e_0 \coloneqq (C \otimes C\op)_0 = C_0 \otimes C_0$ and curvature $h \otimes \eta + \eta \otimes h: C^e \to k$, where $\eta$ is the counit. $C^e$-comodules are bicomodules over $C$. 

We will always consider comodules of $C^e$ in the coderived category of $C^e$, to be defined below.

\begin{remark}
    Note the following subtlety in the setting of curved algebra: A bicomodule is in general neither a left nor a right comodule. If $C$ is curved there is no obvious left or right comodule structure on $C$. Let $h \in C^*$ denote the curvature and denote its action on $C$ by $*$. Then in $C$ by definition $d^2(c) = \{h,c\} = h*c + c*h$, but for a left comodule the differential satisfies $d^2(c) = h*c$. However, $C$ \emph{is} a $C^e$-comodule, as with the curvature $h\otimes 1 + 1 \otimes h$ of $C^e$ we have $d^2(c) = \{h,c\} = (h\otimes \eta + \eta \otimes h)*c$.
\end{remark}

\subsection{Categories as monoids}\label{sect:cat}
We consider dg categories as monoids in bicomodules. We recall the key constructions here and refer to \cite{holstein2022categorical} for further details.
Given a (small) dg category $\mathcal D$ we consider the cosemisimple coalgebra $D_0 \coloneqq \kob {\mathcal D}$ and then $D \coloneqq \oplus_{d, d' \in \Ob \mathcal D} \uHom_{\mathcal D}(d,d')$ is a monoid in the monoidal category of $D_0$-bicomodules with cotensor product. The comodule structure sends any morphism $f$ to $t(f) \otimes f \otimes s(f) \in D_0 \otimes D \otimes D_0$ where $s, t$ are the source and target function.

Recall the cotensor product 
$$M \boxempty_{D_0} N \coloneqq \operatorname{eq}(M \otimes N \rightrightarrows M \otimes D_0 \otimes N)$$ is the equalizer of the left coaction on $N$ and the right coaction on $M$.
Thus a monoid structure $D \boxempty_{D_0} D \to D$ is a composition only defined for $f \otimes g$ if the source of $f$ is the target of $g$.

Note that if $\mathcal D$ is a dg algebra, i.e.\ a dg category with only one object, then $D = \mathcal D$.

A functor $F: \mathcal D \to \mathcal D'$ is represented by a coalgebra map $f_0: D_0 \to D'_0$ (equivalent to the map $F$ induces on objects) and a $D'_0$-bicomodule map $D \to D'$ compatible with the monoid action.

The tensor product of two monoids $D, D'$ in bicomodules is naturally a monoid $D \otimes D'$ in bicomodules over $D_0\otimes D'_0$, and corresponds naturally to the dg category $\mathcal D \otimes \mathcal D'$.

The category $D\Mod$ of left modules over $D$ consists of $D_0$-comodules with an action by the monoid $D$.
Thus this is naturally equivalent to the category of functors from $\mathcal D$ to chain complexes. 
We define the categories of right modules $D\op\Mod$ and bimodules $D^e\Mod$ similarly. In particular, $D$ is naturally a bimodule over itself (also known as the bimodule sending two objects to the hom space between them), but not canonically a left or right module.

We denote by $\De(D)$ the $\infty$-category obtained by localising $D\Mod$ at object-wise quasi-isomorphisms, and by $\HDe(D)$ its homotopy category.
The hom spaces in the natural dg enhancement are denoted by $\RHom_D(-,-)$.

For further details we refer the reader to \cite[Section 2.1]{holstein2022categorical}.

\subsection{Koszul duality of coalgebras and categories}\label{sect:koszul}
We now recall Koszul duality between curved pointed coalgebras and dg categories.

Given a curved pointed coalgebra $C$ we define the dg category (as a monoid in $C_0$-bicomodules)
$\Omega C = (T_{C_0}\overline{C}[1], d)$
where the differential $d$ is induced by the differential, coproduct and curvature on $C$.

The universal twisting cochain $\tau: C \to \Omega C[-1]$ is defined to be the $C_0$-bicomodule map sending $c$ to $[\bar c]$. This is a Maurer-Cartan element in the (curved) convolution algebra $\uHom(C, \Omega C)$.

Dually for a dg category $D$ which is quasi-free (i.e.\ all hom spaces are free $k$-modules) we define
$\Ba D = (T_{D_0} \overline D[-1], d, h)$
where the differential $d$ and curvature $h$ are induced by differential and composition on $D$.
Here the universal twisting cochain may be represented by the projection map that we denote $\tau: \Ba D \to D[-1]$ by abuse of notation. 

By construction we have $(\Ba D)_0 = D_0$ and $(\Omega C)_0 = C_0$ (where these are the cosemisimple coalgebras of objects rather than degree 0 parts).

We have the following Koszul duality result. This is \cite[Theorem 3.40, Corollary 3.41]{holstein2022categorical}. The equivalent modification $\dgCat'$ of $\dgCat$ is also explained in loc.\ cit.

\begin{theorem}\label{thm:algebrakoszul}
    \begin{enumerate}
        \item Let $k$ be a field. Then the adjunction 
        $$\Omega \colon \ptdco \rightleftarrows \dgCat'\colon B$$
        defines a Quillen equivalence of suitably defined model categories.
        \item Let $k$ be a PID. Then $\Omega \dashv \Ba$ is an adjunction between $k$-free pointed curved coalgebras and quasi-free dg categories that induces an equivalence of homotopy categories
$$\Ho((\ptdco)_{\mathsf{fr}}) \rightleftarrows \Ho(\dgCat'_{\mathsf{qfr}})$$
when we localize $\dgCat$ at quasi-equivalences and $\mathsf{ptdCoa}_*$ at maps $f$ such that $\Omega(f)$ is a quasi-equivalence.
    \end{enumerate}
\end{theorem}
The model structure on $\dgCat'$ may be chosen with quasi-equivalences or Morita equivalences as weak equivalences,  we will be interested in quasi-equivalences. The weak equivalences in $\ptdco$ are then just the preimages of quasi-equivalences.

\begin{remark}\label{rem:htpycats}  
The homotopy categories in part (2) of Theorem \ref{thm:algebrakoszul} are just the category obtained by formally inverting all weak equivalences, i.e.\ quasi-equivalences respectively their pre-images.
In general, such a category need not be locally small, however here  $\Ho(\dgCat'_{\mathsf{qfr}})$ is easily seen to be equivalent to $\Ho(\dgCat)$, which is locally small as the homotopy category of a model category (for any ring $k$). Thus $\Ho((\ptdco)_{\mathsf{fr}})$ is also locally small.

The equivalence may be refined, as relative categories, i.e.\ categories with a distinguished class of weak equivalences, provide a model for $(\infty,1)$-categories. 
In fact, even when working over a PID, the bar and cobar functors preserve weak equivalences and the unit and counit of the adjunction are weak equivalences. The reason for assuming freeness is for these functors to be well behaved with respect to homological algebra constructions (all of our applications are in this context). 
We may turn a relative category into an $\infty$-category (quasi-category) by $\infty$-categorical localization of the 1-category at weak equivalences \cite[Section 7.1]{cisinski2019higher}, or into a simplicial category by performing the  hammock localization at weak equivalences  \cite{barwick2012relative, barwick2012characterization}.
These two constructions are equivalent via the homotopy coherent nerve \cite[Proposition 1.2.1]{hinich2015dwyerkanlocalizationrevisited}. As the bar-cobar adjunction induces a strict homotopy equivalence of relative categories by \cite[Corollary 3.41]{holstein2022categorical} it induces a weak equivalence of $(\infty,1)$-categories by \cite[Proposition 7.5]{barwick2012relative}.
\end{remark}

\subsection{Calabi-Yau structures on algebras and categories}\label{sect:cy}
Calabi-Yau structures on dg algebras and dg categories have been studied in a variety of contexts, including algebraic geometry (induced by Serre duality on Calabi-Yau varieties), conformal field theory \cite{costello2007topological, kontsevich2008notes} and representation theory \cite{keller2008calabi}.

There are two different kinds of Calabi-Yau structures.
Let $A$ be a dg algebra or dg category. We call a  dg category \emph{locally proper} if each hom complex is perfect over $k$. 
One often considers \emph{proper} dg categories which moreover have a compact generator, but we shall not need that assumption.

If we consider $A$ as a monoid in $k[\Ob(A)]$-bicomodules 
we can define the $A^e_0$-linear dual $A^* = \uHom_{A^e_0}(A, A^e_0)$ as the dual $A^e_0$-comodule of $A$.
We may write this explicitly by recalling $A = \oplus_{x,y \in \Ob(A)}\uHom_{A}(x,y)$ and dualizing each hom space.
If $A$ is locally proper then $A^*$ has a natural $A_0$-bicomodule structure as each $\Hom$ is perfect $A_0$-bicomodule, and we have an $A^e$-action sending $\alpha \otimes f \otimes \beta \in A \boxempty_{A_0} A^* \boxempty_{A_0} A$ to the element in $A^*$ sending $\gamma \in A$ to $f(\beta \circ \gamma \circ \alpha)$ (as usual the cotensor structure ensures composability).

\begin{definition}
    A \emph{weak proper $n$-Calabi-Yau structure} on a proper dg category $A$ is a cycle in the dual Hochschild chain complex $\HHcx_*(A) \to k[n]$ which induces a weak equivalence $A\to A^*[n]$ of $A$-bimodules.
    A  \emph{proper $n$-Calabi-Yau structure} is a factorization of such a cycle through the dual cyclic chain complex $\HCcx_*(A) \to k[n]$.
\end{definition}

Unravelling the definitions a weak equivalence $A \to A^*[n]$ of $A^e$-bimodules  says that each hom space in $A$ is quasi-isomorphic to its shifted dual, thus this agrees with the definition in \cite{brav2019relative} (except for  the different sign convention).
This is the structure on the derived category of a Calabi-Yau variety.

Kontsevich and Soibelman \cite{kontsevich2008notes} originally defined an $n$-Calabi-Yau structure an a compact $A_\infty$-algebra (or dg algebra) $A$ as a dual cyclic chain $CC_*(A) \to k$ inducing a non-degenerate pairing $A \otimes A \to k[-n]$.  
This is called a \emph{compact Calabi-Yau structure} in \cite{cohencalabi}.

We now denote the envelopping algebra (category) of $A$ by $A^e = A \otimes A\op$. 
We call $A$ \emph{smooth} if the $A^e$-bimodule $A$ is compact as an object in the derived category $\HDe(A^e)$, i.e. if $\RHom_{A^e}(A,-)$ commutes with arbitrary direct sums. This is equivalent to $\RHom_{A^e}(A,-)$ commuting with filtered colimits or with $A$ being perfect as an $A$-bimodule, i.e. being quasi-isomorphic to a retract of a dg $A$-bimodule $M$ such that there is a finite sequence 
\[0 \simeq M_0 \to M_1 \to \ldots \to M_n=M
\] 
with the cone of each map $M _i \to M_{i+1}$ being quasi-isomorphic to a shift of $A^e$.

Let $A^! = \RHom_{A^e}(A, A^e)$ be the derived $A^e$-dual of $A$.

\begin{definition} A \emph{weak smooth $n$-Calabi-Yau structure} on a smooth dg algebra (or dg category) is  a cycle in $\HHcx_n(A)$ inducing a weak equivalence of $A^e$-bimodules $A^! \simeq A[n]$ via the identification $\HH_n(A) \cong  \Hom_{\De(A^e)}(A^!,A[n])$.

A \emph{smooth $n$-Calabi Yau structure} on $A$ is a cycle in the negative cyclic complex $\HNcx_n(A)$ whose image in $\HHcx_n(A)$ is a weak smooth $n$-Calabi-Yau structure.
\end{definition}
We refer to \cite{brav2019relative} for further details. In this reference, proper $n$-Calabi-Yau structures are called \emph{right} $n$-Calabi-Yau structures and smooth $n$-Calabi-Yau structures are called \emph{left} $n$-Calabi-Yau structures. Also note that sometimes what we call a weak $n$-Calabi-Yau structure is called a $n$-Calabi-Yau structure, and a lift to negative cyclic homology is called an \emph{almost exact} $n$-Calabi-Yau structure \cite{herscovich2023cyclic}.

\section{Coalgebras and Calabi-Yau structures}
\subsection{Comodules}\label{sect:comodules}
Next we consider comodules over a curved pointed coalgebra $C$.
We want to consider comodule maps up to homotopy.
This can be expressed in two equivalent ways: 
By considering homotopy comodule maps $M \to \Omega C \boxempty_{C_0} N $, see Definition \ref{def:htpycomodule}, or by localising the category of comodules at a certain class of weak equivalences, see Definition \ref{def:coderived}.

From \cite[Proposition 3.42]{holstein2022categorical} we have the following statement (noting the proof does not depend on $k$ being a field).
\begin{proposition}\label{prop:moduleadjunction}
	\begin{enumerate}
		\item 
	Let $D$ be a dg category viewed as a monoid over $D_0 = k[\Ob(D)]$. There is an adjunction
	\[
	F_D = D\boxempty_{D_0}^\tau -:\Ba(D)\Comod\rightleftarrows A\Mod:\Ba D\boxempty^\tau_{D_0}- = G_D.
	\]
	\item Let $C$ be a curved pointed  coalgebra.
	Then there is an adjunction
	\[
	F_C = \Omega C\boxempty_{C_0}^\tau-:C\Comod\rightleftarrows \Omega(C)\Mod:C\boxempty_{C_0}^\tau- = G_C.
	\]
	\end{enumerate}
\end{proposition}
Analogous adjunctions hold for right modules and comodules.

We will now recall the definition of the functors in the above adjunctions.  We will frequently use the cotensor product over the coalgebra of objects, see  Section \ref{sect:cat} for the definition.
These are objects of the form $D \boxempty_{D_0} N $ etc.\ generalising the usual tensor product in the 1-object case to the situation of dg categories and pointed coalgebras. 

The functor $C\boxempty^\tau_{C_0} -$ is moreover \emph{twisted}, cf.\ \cite{brown1959twisted} by the universal twisting cochain $\tau$ (indicated by the superscript). More precisely,
$C\boxempty^\tau_{C_0} -$
is defined as sending the left $\Omega C$-module $M$ to $C\boxempty_{C_0} M$ with twisted differential 
\[d: c\otimes m \mapsto dc \otimes m +(-)^{|c|} c \otimes dm + d_\tau(c\otimes m)\] where the last 
summand is defined as the composition
$$
d_{\tau} \colon C\boxempty_{C_0} M 
\xrightarrow{\Delta \otimes \id}
C \boxempty_{C_0} C\boxempty_{C_0} M
\xrightarrow{\id \otimes \tau \otimes \id}
C \boxempty_{C_0} \Omega C\boxempty_{C_0} M
\xrightarrow{\id \otimes \lambda}
C\boxempty_{C_0} M
$$
where $\tau: C \to \Omega C$ is the universal twisting cochain (or Maurer-Cartan element) and $\lambda$ denotes the left $\Omega C$-action on $M$. (The same construction may be performed with other twisting cochains.)
The other three functors are defined similarly.

We will sometimes drop the subscript and write $F \dashv G$ when the coalgebra or category we are working over is clear.
Note that $D_0 = C_0$ here, so we may change the subscript of the cotensor product for convenience.

The reader concerned with dg algebras rather than categories may restrict attention to those and replace $D\boxempty_{D_0}^\tau -$ by $D\otimes^\tau -$ etc.\ throughout this paper.

\begin{definition}\label{def:coderived}
  A map $f$ of left $C$-comodules is a \textit{weak equivalence} if $F_C(f)$ is a quasi-isomorphism of left dg $\Omega C$-modules. We call the $\infty$-category of left $C$-comodules localized at weak equivalences the \emph{coderived category} and denote it by $\Dco(C)$.
  Its homotopy category is denoted by $\HDco(C)$. 
\end{definition}
We will see below that our definition is compatible with the usual definition if $k$ is a field. We need not worry about set-theoretic issues as the localization $\Dco(C)$ is equivalent to the familiar derived category $\De(\Omega C)$ because of the following module-comodule Koszul duality.

\begin{theorem}	\label{thm:modulekoszul}
\begin{enumerate}
		\item 	Let $D$ be a dg category. Then
	\[	F_D = D\boxempty_{D_0}^\tau -:\Ba D\Comod\rightleftarrows D\Mod:\Ba D\boxempty^\tau_{D_0}- = G_D	\]
	induces a weak equivalence of $\infty$-categories
 $\Dco(\Ba D) \simeq \De(D)$.
 If $k$ is a field this comes from a Quillen equivalence of suitable model structures.
	\item Let $C$ be a curved pointed  coalgebra. Then
	\[	F_C = \Omega C\boxempty_{C_0}^\tau-:C\Comod\rightleftarrows \Omega C\Mod:C\boxempty_{C_0}^\tau- = G_C	\]
 induces a weak equivalence of $\infty$-categories  $\Dco(C) \simeq \De(\Omega C)$.
If $k$	is a field this comes from a Quillen equivalence for suitable model structures.
	\end{enumerate}
\end{theorem}
Here the model structure on modules is the usual projective model structure where weak equivalences are object-wise quasi-isomorphisms and all objects are fibrant. 
The model category on comodules has weak equivalences as above and all objects are cofibrant.
\begin{proof}
Over a field the necessary model structures are defined in \cite{holstein2022categorical} and this result is \cite[Theorem 3.43]{holstein2022categorical}.

To relate to our definitions we need to observe that our weak equivalences are exactly maps with coacyclic cones, which are the weak equivalences in the model structure on $C\Comod$. 
This  follows as $F$ preserves and reflects weak equivalences which is part of the proof of \cite[Theorem 3.43]{holstein2022categorical}, based on \cite[Theorem 6.3, 6.4]{positselski2011two}. 

  We now prove the equivalence (2) over a general PID.
   We consider the adjunction $F \dashv G$ on the level of relative categories, first observing that $F$ preserves weak equivalences by definition. We claim that for any dg module $M$ over  $\Omega C$ we have a weak equivalence $FGM \to M$. It then follows that $G$ also preserves weak equivalences and that $N \to GFN$ is a weak equivalence for any comodule $N$ by the triangle identity for counits and the 2-out-of-3 property. These ingredients give us an adjunction between relative categories such that the unit and counit maps are weak equivalences. This induces a weak equivalence of $\infty$-categories, see the discussion in Remark \ref{rem:htpycats}.

   We now prove hat $FGM \to M$ is a weak equivalence.
    It is a fairly standard computation that the cocone $K$ of $\Omega C\boxempty^\tau_{C_0} C \boxempty^\tau_{C_0} M \to M$ is acyclic, where the counit is the map $a \otimes c \otimes m \mapsto a\tau(c)m$.
    We follow \cite[Theorem 6.4]{positselski2011two} and repeat the argument for convenience and to convince the reader that it is valid in the multi-object setting over a general ring. 
    We define the decreasing filtration $F_\bullet$ on $K$ 
    by
    $F_{-1} = K$, $F_0 = \Omega C\boxempty^\tau_{C_0} C \boxempty^\tau_{C_0} M$,
    $F_1 = \Omega C\boxempty^\tau_{C_0} 
    \overline{C} \boxempty^\tau_{C_0} M$ and $F_{2} = 0$.
    Then the differential decreases filtration degree by at most 1 and thus induces maps $\delta_1: \Gr^F_1 \to \Gr^F_0$ and $\delta_0: \Gr^F_0 \to \Gr^F_{-1}$.
    Together these give the diagram
    $$0 \to \Omega C\boxempty^\tau_{C_0} 
    \overline{C} \boxempty^\tau_{C_0} M 
\xrightarrow{\delta_1}
\Omega C\boxempty^\tau_{C_0} 
    C_0 \boxempty^\tau_{C_0} M 
    \xrightarrow{\delta_0}
    M[-1] \to 0
    $$
    which is actually a short exact sequence of graded modules: 
    It is the standard resolution of the $\Omega C$-module $M$.
    The fact that we are working with $C_0$-bicomodules does not affect this.
    Indeed, one computes that $\delta_1$ is in fact induced by $a \otimes c \otimes m \mapsto a\tau(c)\otimes s_1 \otimes m - a \otimes s_0 \otimes \tau(c)m$, where $c \mapsto c \otimes s_0$ and $c \mapsto s_1 \otimes c$ are the maps induced by the $C_0$-bicomodule structure of $C$.
    Then the image of $\delta_1$ is exactly the kernel of the map $\delta_0$ induced by the counit map.
    With $\delta_1$ injective we can consider the contractible submodule $L = \cone(F_1 \xrightarrow{\delta_1} \im(\delta_1))$ of $K$, and $K/L$ is also contractible as $\delta_1$ is an isomorphism from $\Gr^F_0/\im(\delta_0)$ to $M[-1]$.
    Together this shows $K$ is an extension of two contractible objects and acyclic.

    Note that $F$ and $G$ preserve weak equivalences, thus they descend to functors on the (co)derived category.
    
   Now (1) follows from (2) and Theorem \ref{thm:algebrakoszul} as $\De(D) \simeq \De(\Omega \Ba D) \simeq \Dco(\Ba D)$.
\end{proof}
The derived and coderived category are linear stable $\infty$-categories and we will consider explicit models for the enrichment in chain complexes of this adjunction and equivalence in Corollary \ref{cor:koszulemoduleenriched}.

To do this we consider an explicit definition of homotopy comodule maps that is analogous to the definition of maps between $A_\infty$-modules. For a left $C$-comodule $M$ denote by  $\la_M: M \to C \otimes M$ the coaction map and by $\mu$ the multiplication on $\Omega C$.
\begin{definition}\label{def:htpycomodule}
    The space of \emph{homotopy comodule maps} between left $C$-comodules $M$ and $N$ is 
    $\uHom^\tau_k(M, \Omega C\boxempty_{C_0}^\tau  N)$ where the superscripts denote that the differential $d^\pi$ is twisted by the coactions on $M$ and $N$, i.e.\ $d^\pi(f)$ is given by
$$
d_M f +(-1)^{|f|} f d_N 
+ (\mu \otimes \id)\circ (\id \otimes \xi \otimes \id) \circ (\id \otimes \la_N) \circ f 
\pm (\mu \otimes \id) \circ (\id \otimes f)\circ (\xi \otimes \id) \circ \la_M$$
with Koszul sign rules.
\end{definition}
This is compatible with Definition \ref{def:coderived} at least for $k$-free modules because of the following corollary:
\begin{corollary}\label{cor:dcohoms}
    Let $M, N$ be $C$-comodules such that $M$ is $k$-free. Then we may compute the hom space in $\HDco(C)$ as $H_0$ of the complex $\uHom^\tau_k(M,\Omega C \boxempty_{C_0}^\tau N)$ from Definition \ref{def:htpycomodule}.
\end{corollary}

\begin{proof}
    From Theorem \ref{thm:modulekoszul} we have $$\Hom_{\HDco(C)}(M,N) \cong \Hom_{\HDe(\Omega C)}(\Omega C \boxempty_{C_0}^\tau M, \Omega C \boxempty_{C_0}^\tau N). $$ 
    But $\HDe(\Omega C)$ is the homotopy category of a model structure on $\Omega C$-modules if $M$ is $k$-free $\Omega C \boxempty_{C_0}^\tau M$ is cofibrant.
    Thus we obtain 
    \begin{align*}
        \Hom_{\HDe(\Omega C)}(\Omega C \boxempty_{C_0}^\tau M, \Omega C \boxempty_{C_0}^\tau N) 
        &\cong H_0(\uHom_{\Omega C}(\Omega C \boxempty_{C_0}^\tau M, \Omega C \boxempty_{C_0}^\tau N)) \\
        &\cong H_0(\uHom^\tau_k(M,  \Omega C \boxempty_{C_0}^\tau N))       
    \end{align*}
    where the superscript $\tau$ of $\Hom$ stands for the twisting induced by the twisting of the cotensor product.
    The claim follows by unravelling the differential.
\end{proof}
    This enrichment provides a (partial) dg enhancement of the adjunction and equivalence of homotopy categories from Theorem \ref{thm:modulekoszul} even if $k$ is not a field.
 \begin{corollary}\label{cor:koszulemoduleenriched}
Defining $\RHom_{C}(N, N')$ as $\uHom^\tau_k(N,  \Omega C \boxempty_{C_0}^\tau N')$ we have 
natural quasi-isomorphisms
$$\RHom_{\Omega C}(F N, M) \simeq \RHom_{C}(N, GM),$$
$$LF: \RHom_C(N, N') \simeq \RHom_{\Omega C}(FN, FN'),$$
and
$$RG: \RHom_{\Omega C}(M, M') \simeq \RHom_{C}(GM, GM')$$
for $N, N \in C\Comod$ and $M, M'\in \Omega C\Mod$ with $M, N$ assumed $k$-free. 
 \end{corollary}
 This shows one can compute a dg enhancement of   $\Dco(C)$ by $\RHom_{C}$ if the domain is $k$-free, which explains the somewhat abusive name.
 
 \begin{proof}
     The second statement follows just as in the proof of Corollary \ref{cor:dcohoms}. For the first statement we then use $M' = FN'$ together with invariance of $\RHom_{\Omega C}$ under weak equivalences. Similarly, the final statement follows by taking $M = FN$ and $M' = FN'$.
 \end{proof}

\subsection{Proper coalgebras}\label{sect:coproper}
In analogy with the case of algebras and categories, Calabi-Yau structures on coalgebras should be cycles in negative cyclic homology
identifying a coalgebra with a suitable dual under the appropriate notion of weak equivalence. As the dual of a comodule is not necessarily a comodule, we first define a suitable class of coalgebras. These are called proper coalgebras and we define and analyze them in this section. We start with the following key computation.
\begin{lemma}\label{lem:aec} 
There are natural weak equivalences   $F_{C^e} C = (\Omega C)^e \boxempty_{C_0^e}^{\tau^e} C \simeq \Omega C $ and $G_{A^e} A = (\Ba A)^e \boxempty_{A_0^e}^{\tau^e} A \simeq \Ba A$. 
\end{lemma}
\begin{proof}
 For the first statement we have $(\Omega C)^e \boxempty_{C_0^e}^\tau C  = F_{C^e}C$ by definition, see Proposition \ref{prop:moduleadjunction}. But $(\Omega C)^e \boxempty_{C_0^e}^{\tau^e} C \simeq \Omega C \boxempty_{C_0}^\tau C \boxempty_{C_0}^\tau \Omega C$ is also equivalent to $F_CG_C(\Omega C)$ as both comodules have the same underlying graded object and the same differential, compare the proof of \cite[Corollary 5.7]{guan2023hochschild}.
    Note that the universal twisting cochain $C^e \to \Omega C^e$ takes the form $\tau^e = \tau_C \otimes 1 + 1 \otimes \tau_{C\op}$, thus the twisting on the complex $\Omega C^e \boxempty_{C^e_0}^{\tau^e} C$ has two summands, just like for $G_CF_C(\Omega C)$. Since $G_CF_C(\Omega C)$ is weakly equivalent to $\Omega C$ by  Theorem \ref{thm:modulekoszul}, we are done.
 
    The second statement can be shown in the same way.
\end{proof}
\begin{definition}\label{def:cohh}
    Let $C$ be a pointed curved coalgebra. 
    We define its \emph{coHochschild complex} $\coHHcx_*(C)$
    to have underlying $k$-module the same as
    $$C \boxempty_{C^e_0} \Omega C $$
    with differential
    induced by the differentials on $\Omega C$ and $C$ plus the two-sided twisting
    $$c(c_1|\dots|c_n)
    \mapsto
    \sum \pm c^{(1)}(c^{(2)}|c_1|\dots|c_n) +
    \sum \pm c^{(2)}(c_1|\dots|c_n|c^{(1)})
    $$
    with Koszul sign convention and Sweedler notation for the coproduct of $C$. As this is a two-sided version of the twist used in Proposition \ref{prop:moduleadjunction} we write the complex as
    $$C \boxempty_{C^e_0}^\pi \Omega C.$$
    See \cite[Section 3.5]{rivera2024cyclichomologycategoricalcoalgebras} or \cite{rivera2022algebraic} for more details on the construction. We define the \emph{coHochschild homology} of $C$, denoted by $\coHH_*(C)$, as the homology of $\coHHcx_*(C)$.

By \cite[Section 5.3]{rivera2024cyclichomologycategoricalcoalgebras} $\coHHcx_*(C)$ is a mixed complex and thus we may define the \emph{cocyclic complex} $\coHCcx_*(C) = (\coHHcx_*(C)[u^{-1}], b+uB)$ as well as the
\emph{negative cocyclic complex} $\coHNcx_*(C) = (\coHHcx_*(C)[[u]], b+uB)$.The \emph{cocyclic homology} $\coHC_*(C)$ and \emph{negative cocyclic homology} $\coHN_*(C)$ of $C$ are just the homology of these complexes.
\end{definition}

\begin{remark}\label{rem:hochschild}
We may write the coHochschild complex as \[C \boxempty_{C^e}(C^e\boxempty_{C^e_0}^{\tau^e} \Omega C).\]
If $k$ is a field we note that $\coHHcx_*(C) \simeq C \boxempty_{C_0}^\pi \Omega C$ is an explicit model for the derived cotensor product $C \boxempty_{C^e}^L C$ using $C \simeq C^e \boxempty_{C_0^e}^{\tau^e} \Omega C^e$ from Lemma \ref{lem:aec} where we consider the derived cotensor product as a functor on the coderived category, see \cite[Section 4.7]{positselski2011two}. 
    
\end{remark}
This notion of Hochschild and cyclic homology is compatible with the notion for dg categories:
\begin{proposition}\label{prop:cohhishh}
       For any pointed curved coalgebra $C$, there is a natural quasi-isomorphism of mixed complexes \[\coHHcx_* (C)\xrightarrow{\simeq} \HHcx_*(\Omega C).\]
      Consequently, we have induced quasi-isomorphisms $\coHCcx_*(C) \simeq \HCcx_*(\Omega C)$, $\coHNcx_*(C) \simeq \HNcx_*(\Omega C)$, $\coHH_*(C) \cong \HH_*(\Omega C)$ and $\coHN_*(C) \cong \HN_*(\Omega C)$.
\end{proposition}
\begin{proof}
    This is \cite[Theorem 5.3.5]{rivera2024cyclichomologycategoricalcoalgebras}.
\end{proof}

We now consider the key finiteness condition on coalgebras. Recall that we call an object $K$ in a triangulated category (or stable $\infty$-category) \emph{compact} if $\Hom(K, -)$ commutes with arbitrary direct sums. 

\begin{definition}
We say $C$ is \emph{proper} if it is compact as a $C^e$-comodule.
\end{definition}

The most important class of examples comes from the following lemma. 
\begin{lemma}\label{lem:properfinite}
    If a pointed curved coalgebra $C$ is weakly equivalent as a $C^e$-comodule to a direct summand of a $C^e$-comodule whose underlying graded $k$-module is finitely presented, then $C$ is proper. In particular any pointed curved coalgebra of finite rank over $k$ is proper.
\end{lemma}
\begin{proof}
It suffices to show that any curved comodule $N$ whose underlying graded is free of finite rank is compact as the property is stable under weak equivalence and direct summands.
We consider a direct sum $\oplus_i X_i$ in $C\Comod$.
Using Corollary \ref{cor:dcohoms}
we may compute
$\Hom_{\HDco(C)}(N, \oplus_i X_i)$
as
$$H^0(\uHom^\tau(N, \Omega C \boxempty_{C_0}^\tau  (\oplus_i X_i)) \simeq  \oplus_i H^0(\uHom^\tau(N, \Omega C \boxempty_{C_0}^\tau  X_i)$$ using finite-presentation of $N$.
Applying Corollary \ref{cor:dcohoms} again we see that $$\Hom_{\HDco(C)}(N, \oplus_i X_i) = \oplus_i \Hom_{\HDco(C)}(N,  X_i).\qedhere$$
\end{proof}
\begin{remark}
   The converse holds if $k$ is a field: Any compact $C$-comodule is isomorphic in the coderived category to a retract of a finite-dimensional comodule.
   
Let $N$ be compact and write it as a union of finite-dimensional subcomodules $N_i$. 
We recall that filtered colimits are homotopy colimits in a combinatorial model category generated by cofibrations with $\aleph_0$-small source and target, see \cite[Proposition 7.3]{dugger2001combinatorial} together with \cite[Proposition A.1.2.5]{lurie2009higher}.
Furthermore $C^e\Comod$ is of this form as the category of comodules is generated by finite-dimensional comodules.
Thus we  have $C \simeq \hocolim(N_i)$ and by compactness $\id_N$ factors, up to weak equivalence, through a finite subcomodule, presenting $N$ as a homotopy summand of a finite-dimensional comodule.
\end{remark}

\begin{proposition}\label{prop:smoothproper}
A pointed curved coalgebra $C$ is proper if and only if the dg category $\Omega C$ is smooth. Conversely, a dg category $A$ is smooth if and only if the pointed curved coalgebra $\Ba A$ is proper.
\end{proposition}

The proof needs the following observation:
\begin{lemma}\label{lem:omegaop}
There is a natural quasi-isomorphism of dg algebras $(\Omega C)^e \simeq \Omega(C^e)$.
\end{lemma}
\begin{proof}
    As $\Omega$ is quasi-strong monoidal \cite{holstein2023enrichedkoszuldualitydg} it suffices to show there is an isomorphism $s: \Omega(C\op) \cong \Omega(C)\op$. Define $s$ by reversing the order of tensor factors 
    \[ (c_1 | c_2 | \dots | c_n) \mapsto \pm (c_n | c_{n-1} | \dots c_1).\]
    On the underlying graded algebras this is an isomorphism $T\bar C\op[1] \cong (T\bar C[1])\op$. It remains to check compatibility with the differential. 
    It is clear that $s$ commutes with the component of the differential induced by $d_C$ and since $s \circ \Delta\op = \Delta $ it commutes with the component induced by the coproduct. 
    The component induced by the curvature is also unaffected by $s$.
\end{proof}
\begin{proof}[Proof of Proposition \ref{prop:smoothproper}]
    Koszul duality of modules (Theorem \ref{thm:modulekoszul}) together with Lemma \ref{lem:omegaop} gives a weak equivalence $\Dco(C^e) \simeq \De((\Omega C)^e)$. This identifies compact objects. It remains to show that it identifies the bicomodule $C$ with the bimodule $\Omega C$. 
    But this is exactly Lemma \ref{lem:aec}.

Now let us conversely assume $A$ is a given dg category. It is smooth if and only if $\Omega \Ba A$ is smooth, as being smooth is invariant under quasi-equivalence. This is the case if and only if $\Ba A$ is proper.
\end{proof}
\begin{corollary}\label{cor:properwelldefined}
    Let $B$ and $C$ be weakly equivalent curved pointed coalgebras. 
    Then $C$ is proper if and only if $B$ is.
\end{corollary}
\begin{proof}
    If $C \simeq B$ then $\Omega C \simeq \Omega B$ by Theorem \ref{thm:algebrakoszul}, and thus $C$ is proper if and only if $B$ is by Proposition \ref{prop:smoothproper}.
\end{proof}

\begin{remark}
    One difference to the algebra-setting is that while weak equivalences of bimodules and dg algebras or dg categories are defined in the same way as quasi-isomorphisms (or quasi-equivalences), the weak equivalences of curved coalgebras and their comodules are defined quite differently. While it is clear from the definition, that a curved pointed coalgebra of finite rank is proper, it does not follow that any proper pointed curved coalgebra is weakly equivalent to a pointed curved coalgebra of finite rank.
\end{remark}

We now turn to define an appropriate dual for a proper coalgebra. For any left comodule $M$ we have the natural maps
\[ M^* \xrightarrow{\alpha} \uHom(M,C) \xleftarrow{\beta} M^* \otimes C,\]
given by $\alpha(f)=(\id \otimes f) \circ \gamma$ where $\gamma: M \to C \otimes M $ is the coaction,
and
$\beta(g \otimes c): m \mapsto g(m)c$. 
When $M$ is finitely presented, the map $\beta$ is an isomorphism, but in general it is not possible to invert $\beta$ even if $C$ is proper.

If $M$ is finitely presented we define the coaction $\gamma^*: M^* \to M^* \otimes C$ as $\beta^{-1} \circ \alpha$ and have the following lemma that applies in many important examples:

\begin{lemma}\label{lem:dualcomodule}
    For any  left $C$-comodule $M$ finitely presented over $k$ the coaction $\gamma^*$ makes $M^*$ into a right $C$-comodule. In particular if $C$ itself is finitely presented over $k$ then $C^*$ has a natural structure of a $C$-bicomodule.
\end{lemma}
\begin{proof}
The conditions are straightforward to check. For the second observation we just consider $C$ as $C^e$-comodule.
\end{proof}

To define a more general dual $C^\vee$ we consider the Koszul image of $(\Omega C)^!$, where $A^! = \RHom_{A^e}(A,A^e)$ is well-behaved for $A = \Omega C$ as $A$ is smooth by Proposition \ref{prop:smoothproper}.
\begin{definition}
    For a proper pointed curved coalgebra $C$, define
    \[C^\vee \coloneqq C^e \boxempty^{\tau^e}_{C^e_0} \uHom^
{\tau^e}(C, \Omega C^e).\]
\end{definition}
Here the differential is explicitly given by combining the twisted cotensor product from Proposition \ref{prop:moduleadjunction} with the twisted hom space from Definition \ref{def:htpycomodule}. We think of $C^\vee$ as the linear dual of $C$. This is literally true if $C$ is of finite rank, see Lemma \ref{lem:computeadjoint} below. Another sense in which $C^\vee$ is dual to $C$ is revealed in Lemma \ref{lem:cptadjunction} below.

\begin{lemma}\label{lem:omegacshriek}
       For a proper curved coalgebra $C$ we have natural weak equivalences $C^\vee  \simeq G((\Omega C)^!)$ and $F(C^\vee) \simeq (\Omega C)^!$.
\end{lemma}
\begin{proof} 
The first weak equivalence follows since, for $A = \Omega C$, we have 
$$A^! \simeq \RHom_{A^e}(A, A^e) \simeq \uHom_{A^e}(A^e \boxempty^{\tau^e}_{A^e_0} C, A^e) \simeq \uHom^{\tau^e}(C, A^e)$$ using Lemma \ref{lem:aec} and the free-forgetful adjunction. The second weak equivalence in the statement is an immediate consequence.
\end{proof}
If $C$ has finite rank over $k$, we have a more familiar description of the dual.
  \begin{lemma}\label{lem:computeadjoint}
Let $E$ be a $C$-bicomodule weakly equivalent to $C$ that is free of finite rank over $k$ and consider the linear dual $E^*$ with its $C$-comodule structure from Lemma \ref{lem:dualcomodule}.
Then $F_{C^e}(E^*)  \simeq (\Omega C)^!$ and $G_{C^e}((\Omega C)^!) \simeq E^*$. In particular, we have a weak equivalence of $C$-bicomodules $E^* \simeq C^\vee$.
\end{lemma}
The lemma applies in particular if $C$ itself has finite rank over $k$.
This result also explains why we want to think of $C$ as proper rather than smooth: it is an object weakly equivalent to its $k$-linear dual.

\begin{proof}
Write $A = \Omega C$. As $G$ preserves weak equivalences of $C$-bicomodules by Theorem \ref{thm:modulekoszul}, we have $$A^e\boxempty_{C_0^e}^\tau E \simeq F(E) \simeq F(C) \simeq A,$$
where the last weak equivalence follows from Lemma \ref{lem:aec}. As $A^e\boxempty_{C_0^e}^\tau E$
 is $A^e$-free we can compute \[A^! = \RHom_{A^e}(A, A^e) \simeq \uHom_{A^e}(E\boxempty_{C_0^e}^\tau A^e, A^e) \simeq \uHom_{A^0_e}^\pi(E, A^e)\] using the free-forgetful adjunction (for $A^e$-modules in $A_0^e$-comodules) and this is $A^e \boxempty_{A^e_0}^{\tau^e} E^*$ as $E$ is free of finite rank.
 Thus \[C^{\vee} \simeq G(A^!) \simeq C^e \boxempty^{\tau^e}_{A^e_0} A^e \boxempty^{\tau^e}_{A^e_0} E^* \simeq E^*,\] as desired. The second statement follows as $F \dashv G$ induces an equivalence of homotopy categories. The final statement is just the definition of $C^\vee$. 
\end{proof}

If $A$ is a smooth dg category there is an adjunction 
$$A^! \otimes_k^L -
\colon \De(k) \rightleftarrows \De(A^e) \colon  A \otimes_{A^e}^L -,$$
see \cite[Example 2.10]{brav2019relative}. 
This can be seen at the level of dg categories by using the tensor-hom adjunction for the $(k,A^e)$-bimodule $A^!$ and using $\RHom_{A^e}(A^!, -) \simeq A \otimes^L_{A^e}- $ by compactness of $A$ in $\De(A^e)$. We translate this adjunction across Koszul duality to obtain the following.

\begin{lemma}\label{lem:cptadjunction}
For any proper pointed curved coalgebra $C$, we have an adjunction between homotopy categories
$$C^\vee  \otimes_k -
\colon \HDco(k) \rightleftarrows \HDco(C^e) \colon C \boxempty_{C^e}^L - .$$
\end{lemma} 
This is induced by a Quillen adjunction when $k$ is a field, otherwise we have to reinterpret $C\boxempty^L_{C^e} -$ as $(C^e\boxempty^{\tau^e}_{C^e_0} \Omega C)  \boxempty_{C^e} -$, which preserves weak equivalences and thus descends to localizations. Note that $\HDco(k)$ is just equivalent to $\HDe(k)$.
\begin{proof}
    We can deduce this from the adjunction for algebras, which provides a natural isomorphism
    \begin{align}\label{eq:algadj}
        \Hom_{\HDe((\Omega C)^e)} ((\Omega C)^! \otimes_k^L M, N) &\cong \Hom_{\HDe(k)}(M, \Omega C \otimes_{
    (\Omega C)^e}^L N).
    \end{align}
We apply this when $N$ is the Koszul image $(\Omega C)^e \cotensorte N'$ of an arbitrary $C^e$-comodule $N'$. Then we have
   the natural weak equivalence
    $$(C^e\boxempty^{\tau^e}_{C^e_0} \Omega C)  \boxempty_{C^e} N' \xrightarrow{\sim} \Omega C \boxempty^\pi_{C^e_0} M 
    \xrightarrow{\sim} \Omega C \otimes_{(\Omega C)^e}((\Omega C)^e \cotensorte N'),$$
    which identifies the right hand side of Equation \ref{eq:algadj} with
    $\Hom_{\HDe(k)}(M, C\boxempty_{C^e}^L N')$.

    It remains to identify the left hand side of Equation \ref{eq:algadj} with $\Hom_{\HDe(k)}(C^\vee \otimes M, N')$.
    We rewrite $\Hom_{\HDe((\Omega C)^e)} ((\Omega C)^! \otimes_k^L M, N)$ as 
    \begin{align*}
        \Hom_{\HDe((\Omega C)^e)} ((\Omega C)^! \otimes_k^L M, FN') 
        &\cong \Hom_{\HDco(C^e)} (G((\Omega C)^! \otimes_k^L M), GFN')\\
    &\cong \Hom_{\HDco(C^e)} (G((\Omega C)^! \otimes_k^L M), N'),
       \end{align*}
       using the Koszul adjunction and equivalence.
    Then the source is $G((\Omega C)^! \otimes M) \simeq G((\Omega C)^!) \otimes M \simeq C^\vee \otimes M$, by Lemma \ref{lem:omegacshriek} and the lemma is proven.
\end{proof}

\subsection{Proper Calabi-Yau structures on coalgebras}\label{sect:propercy}
We now define a notion of proper Calabi-Yau structure on a proper pointed curved coalgebra through the coHochschild complex.

\begin{lemma}\label{lem:hochschildmap}
Let $C$ be proper pointed curved coalgebra.
Then a cycle $\beta\in \coHHcx_n(C)$ induces a natural $C$-bicomodule map $\phi(\beta): C^\vee \to C[n]$ in $\Dco(C^e)$. Conversely, any map $C^\vee \to C[n]$ in $\Dco(C^e)$ gives rise to a cycle in $\coHHcx_n(C)$.
\end{lemma}
\begin{proof}
This follows from Lemma \ref{lem:cptadjunction} and Lemma \ref{lem:aec}.
We define $\phi$ as the composition 
\begin{align*}
    \coHHcx_*(C) &\simeq \RHom(k, C \boxempty^\pi_{C^e_0} \Omega C) 
    \simeq \RHom(k, C \boxempty_{C^e}(C^e \boxempty^{\tau^e}_{C^e_0} \Omega C)) \\
    & \stackrel {(*)} \simeq \RHom_{C^e}(C^\vee, C^e \cotensorte \Omega C) \simeq \RHom_{C^e}(C^\vee, C).
\end{align*}
Here for $(*)$ we use a dg-enhancement of the adjunction, which follows with the same computation as Lemma \ref{lem:cptadjunction} as $C^\vee$ is $k$-free and we may use Corollary \ref{cor:koszulemoduleenriched}. Also by Corollary \ref{cor:koszulemoduleenriched}
 the element in $\RHom_{C^e}(C^\vee, C)$ defines a map in $\Dco(C^e)$.
\end{proof}
There is of course also an indirect way of showing the lemma, by considering the composition of isomorphisms
\[\coHH_n(C) \cong \HH_n(\Omega C) \cong \Hom_{\HDe(A^e)}(A^!, A[n]) \cong \Hom_{\HDco(C^e)}(C^\vee, C[n]),\]
where the first map is Proposition \ref{prop:cohhishh}, the middle map is a standard isomorphism and the last map Theorem \ref{thm:modulekoszul} together with Lemma \ref{lem:aec}.
    
\begin{remark}\label{rem:cycdual}
    Let $C$ be a proper pointed curved coalgebra of finite rank and $\beta \in \coHHcx_n(C)$.
    Then we may describe the induced map $\phi(\beta): C^* \to C[n]$ in $\HDco(C^e)$ explicitly as follows. We have isomorphisms of complexes
$$\coHHcx_*(C) \cong \uHom^\pi(C^*, A) \cong  \uHom_{C^e}(C^*, C^e \boxempty_{C^e_0}^\tau A),$$ 
where the first map is induced by the natural map $C\boxempty_{C_0} \Omega C \to \uHom_{C_0}(C^*, \Omega C)$ on underlying graded objects, using that $C$ is of finite rank.
The second map is obtained from the cofree adjunction for $C^e$-comodules
and explicit comparison of the differential, which is induced by the left and right action of the twisting cochain in both cases. Thus the natural map from cycles in $\uHom_{C^e}(C^*, C^e \boxempty_{C^e_0}^\tau A)$ to $\Hom_{\HDco(C^e)}(C^*, C^e \otimes A) \cong \Hom_{\Dco(C^e)}(C^*, C)$ associates to any  cycle in $\coHHcx_*(C)$ a map $C^* \to C$ in $\HDco(C^e)$.
\end{remark}

\begin{definition}\label{def:cyproper}
    Let $C$ be proper pointed curved coalgebra. A cycle $\beta \in \coHHcx_n(C)$ is called \emph{non-degenerate} if it induces a weak equivalence $C^\vee \simeq C[n]$ in $\Dco(C^e)$.

    A non-degenerate cycle $\beta \in \coHHcx_n(C)$ is called a  \emph{weak proper $n$-Calabi-Yau structure}.

     A \emph{proper $n$-Calabi-Yau structure} on $C$ is a cycle $\nu \in \coHNcx_n(C)$ in the negative cyclic complex whose image in $\coHHcx_n(C)$ is a weak proper $n$-Calabi-Yau structure. 
\end{definition}
Our main statement is the following.
\begin{theorem}\label{thm:main}
A proper curved pointed coalgebra $C$ has a (weak) proper $n$-Calabi-Yau structure if and only if $\Omega C$ has a (weak) smooth $n$-Calabi-Yau structure.

Conversely a smooth dg category $A$ has a (weak) smooth $n$-Calabi-Yau structure if and only if $BA$ has a (weak) proper $n$-Calabi-Yau structure. 
\end{theorem}

\begin{proof}
We will consider the following diagram of chain complexes with $A = \Omega C$.
    $$\begin{tikzcd}
  \HHcx_*(A) \arrow[r, "\simeq"] \arrow[d, "\simeq"]
    & \coHHcx_*(C) \arrow[d, "\simeq"]    \arrow[bend left = 90]{ddd}{\phi}
    \\
    \uHom(k, A \otimes_{A^e}  (A^e \cotensorte C)) \arrow[d, "\simeq"]  \arrow[r, "\simeq"]
    & \uHom(k, C \boxempty_{C^e} (C^e \cotensorte A)) \arrow[d, "\simeq"] \\
     \RHom_{A^e}(A^!, A^e \cotensorte C) \arrow[d, "\simeq"] 
     & \RHom_{C^e}(C^\vee, C^e \cotensorte A) \arrow[d, "\simeq"] 
     \\
     \RHom_{A^e}(A^!, A)  \arrow{r}{RG_{C^e}}[swap]{\simeq}
     & \RHom_{C^e}(C^\vee, C)
\end{tikzcd}$$
The top horizontal map is a quasi-isomorphism by Proposition \ref{prop:cohhishh}.
The vertical maps are a factorization of the map $\phi$ from Lemma \ref{lem:hochschildmap}.
The left vertical quasi-isomorphisms are the analogous computations for algebras, where we use that we may write $\HHcx_*(A)$ as $A \otimes^\pi C$ as well as $C \otimes^\pi A$.
The functor $G_{A^e}$ is the right adjoint in the Koszul duality adjunction and we use Lemma \ref{lem:omegacshriek} and Corollary \ref{cor:koszulemoduleenriched} ($C^\vee$ is free over $k$). The bottom rectangle commutes as the adjunctions correspond under Koszul duality.

As Koszul duality for modules and comodules induces an equivalence of homotopy categories invertible maps in the bottom row are identified.
Thus a cycle in  $\coHHcx_*(C)$ gives rise to weak proper Calabi-Yau structure on $C$ if its image in $\HHcx_*(A)$ gives rise to a weak smooth Calabi-Yau structure on $A$. 

This shows the result for weak Calabi-Yau structures. 
For Calabi-Yau structures we just extend the diagram at the top by the quasi-isomorphism $\coHNcx_*(C)
\simeq \HNcx_*(\Omega C)$.

 For the converse, note that $A$ has a smooth $n$-Calabi-Yau structure if and only if $\Omega \Ba A$ does, which by the first statement is the case if and only if $\Ba A$ has a proper $n$-Calabi-Yau structure.
\end{proof}
The proof of the theorem immediately gives the following corollary.
\begin{corollary}
    If we consider Calabi-Yau structures up to equivalence, i.e.\ their classes in $\coHN_n(C)$ and $\HN_n(\Omega C)$, then there is a bijective correspondence of proper $n$-Calabi-Yau structures on $C$ and smooth $n$-Calabi-Yau structures on $\Omega C$.
\end{corollary}
\subsection{Smooth Calabi-Yau structures on coalgebras}\label{sect:smoothcy}
We now consider the dual case.
To consider the Koszul dual notion to a locally proper dg category we observe that $A$ is locally proper if $A \otimes^L_{A^e} -: \HDe(A^e)\to \HDe(k)$ has a right adjoint of the form $A^* \otimes^L_k -$ \cite[Section 2.2]{brav2019relative}. 
Indeed, the adjunction on homotopy categories suffices to show that $A^*\otimes -$ is naturally weakly equivalent to $\RHom(A,-)$ and $A$ is locally proper.

\begin{definition}\label{def:cysmooth}
    We say $C \in \ptdco$ is \emph{weakly smooth} if 
    $$C \boxempty^L_{C^e} -: \HDco(C^e) \to \HDco(k)$$ has a right adjoint of the form $N \mapsto C^! \otimes N$ for some $C$-bicomodule $C^!$. 

    Here again if $k$ is not a field then we have to replace $C \boxempty^L_{C^e} -$ by the explicit cotensor product $(C^e \boxempty^{\tau^e}_{C^e_0}\Omega C) \boxempty_{C^e} -$ which preserves weak equivalences and therefore induces a map of coderived categories.
\end{definition}

\begin{remark}\label{rmk:weaklysmooth}
    We will see below that the notion defined in \ref{def:cysmooth} is Koszul dual to the notion of (local) properness, thus by analogy with the dual situation we called this property smoothness in the first version of this paper. However, note that this terminology is in conflict with the notion of a \textit{smooth pseudocompact algebra}, \cite{van2015calabi, herscovich2023cyclic}, defined as a pseudocompact algebra $A$ such that $A$ is perfect as a $A^e$-module, that is, it is contained in the triangulated subcategory generated by $A^e$.   
    Thus, it is reasonable to call a coalgebra smooth if its dual pseudocompact algebra is smooth, as in \cite{booth2025nonsmoothcalabiyaustructuresalgebras}. 
    In contrast to the situation for algebras, being perfect as a $C^e$-comodule is not equivalent to having a right adjoint as a functor from $C^e$-modules to $k$-modules.

    To see that being weakly smooth is indeed weaker than being smooth we may check it for the Koszul dual notion: being in the triangulated subcategory generated by $C^e$ corresponds to being in the triangulated subcategory of $\Omega C^e$-modules generated by $k$, which is a stronger condition than being locally proper. See also \cite[Section 5.4]{booth2025nonsmoothcalabiyaustructuresalgebras}.
\end{remark}
\begin{proposition}\label{prop:cshriek}
    $C \in \ptdco$ is {weakly smooth} if and only if $\Omega C$ is a locally proper dg category.
    Moreover, we have $$C^! \simeq C^e\boxempty_{C_0^e}^{\tau^e} A^* \simeq G(A^*).$$ 
\end{proposition}
As local properness is invariant under quasi-equivalence, it follows that  $A$ is locally proper if and only if $\Ba A$ is weakly smooth. 
\begin{proof}
For simplicity, write $A=\Omega C$. 
Koszul duality 
together with Lemma \ref{lem:aec}
identifies $A\otimes_{A^e}^L- $ with $C \boxempty_{C^e}^L -$. Then the existence of the right adjoint in the two cases is equivalent
and we have 
$$A\otimes^L_{A^e} - \colon \HDe(A^e)\rightleftarrows \HDe(k)
\colon A^* \otimes_k^L -$$
if and only if
$$C\boxempty^L_{C^e} - \colon \HDco(C^e)\rightleftarrows \HDe(k)
\colon C^! \otimes_k^L -$$
where $C^!$ respectively $A^*$ is just the image of $k$ under the right adjoint.
The formula for $C^!$ follows as $G(A^*) = C^e \cotensorte A^*$.
    \end{proof}
 \begin{remark}\label{rmk:smoothpseudocompact}
     Note that $C^!$ is the predual of $\RHom_{C^e}(C^e, C)$
  as 
  $$\uHom(C^e\boxempty_{C^e_0}^\tau A^*, k) \simeq \uHom_{C^e_0}(C^e, A) = \uHom_{C^e}(C^e, C^e \boxempty^\tau_{C_0^e} A) \simeq \RHom_{C^e}(C^e, C).$$
It follows that $C^!$ is indeed the analogue of $A^! = R\uHom_{A^e}(A, A^e)$ in $C$-bicomodules.
 \end{remark}

\begin{lemma}\label{lem:smoothcymap}
    A cycle $\eta: \coHHcx_*(C) \to k[-n]$  or $\coHCcx_*(C) \to k[-n]$ for a weakly smooth pointed curved coalgebra $C$ induces a map $C[n]\to C^!$.
\end{lemma}
 \begin{proof}
     This follows by adjunction in Definition \ref{def:cysmooth} as we have 
     \begin{align*}
         \RHom(\coHHcx_*(C), k) & \simeq \RHom(C \boxempty^\pi_{C^0_e} \Omega C, k)
         \simeq \RHom(C \boxempty_{C^e} (C^e \boxempty^{\tau^e}_{C^0_e} \Omega C), k) \\
         &\stackrel{(*)}\simeq \RHom_{C^e}(C^e \cotensorte \Omega C, C^!) \simeq \RHom_{C^e}(C, C^!)
     \end{align*}
     where for $(*)$ we used a dg enhancement of the adjunction in Definition \ref{def:cysmooth}.
     This automatically exists as we may translate the adjunction to an adjunction on homotopy categories of dg modules by Proposition \ref{prop:cshriek}. Thus $\Omega C$ is locally proper and the adjunction of dg modules has a dg enhancement. 
     Then we may use Corollary \ref{cor:koszulemoduleenriched} to enhance the adjunction of comodules as $C^e \cotensorte\Omega C$ is $k$-free.
 \end{proof}
 \begin{definition}
 We say a weakly smooth pointed curved coalgebra $C$ has a \emph{smooth $n$-Calabi-Yau} structure if there is a cycle $\eta: \coHCcx_*(C) \to k[-n]$ in the dual of the cocyclic homology complex which induces a weak equivalence $C[n] \to C^!$.
  We say a weakly smooth pointed curved coalgebra $C$ has a \emph{weak smooth $n$-Calabi-Yau} structure if there is a cycle $\eta: \coHHcx_*(C) \to k[-n]$ in the dual of the coHochschild complex which induces a weak equivalence $C[n] \to C^!$.
 \end{definition}

\begin{theorem}\label{thm:dualmain}
    A locally proper dg category $A$ has a (weak) proper $n$-Calabi-Yau structure if and only if $\Ba A$ has a (weak) smooth $n$-Calabi-Yau structure.
    
    A weakly smooth pointed curved coalgebra $C$ has a (weak) smooth $n$-Calabi-Yau structure if and only if $\Omega C$ has a (weak) proper $n$-Calabi-Yau structure.
\end{theorem}
\begin{proof}
    Let  $C = \Ba A$.
    The proof works just like the proof of Theorem \ref{thm:main}, using Lemma \ref{lem:smoothcymap} and following the maps induced by the Hochschild cycle and the coHochschild cycle:

        $$\begin{tikzcd}
  \uHom(\HHcx_*(A), k) \arrow[r, "\simeq"] \arrow[d, "\simeq"]
    & \Ho(\coHHcx_*(C),k) \arrow[d, "\simeq"] \\
    \uHom(A \otimes_{A^e}  (A^e \cotensorte C), k) \arrow[d, "\simeq"]  \arrow[r, "\simeq"]
    & \uHom(C \boxempty_{C^e} (C^e \cotensorte A), k) \arrow[d, "\simeq"] \\
     \RHom_{A^e}(A^e \cotensorte C, A^*) \arrow[d, "\simeq"] 
     & \RHom_{C^e}(C^e \cotensorte A, k \cotensorte A) \arrow[d, "\simeq"] 
     \\
     \RHom_{A^e}(A, A^*)  \arrow{r}{RG_{A^e}}[swap]{\simeq}
     & \RHom_{C^e}(C, C^!) 
\end{tikzcd}$$

Again a dual Hochschild cycle induces a weak equivalence $A \to A^*$ if and only if its image in $\coHH_*(C)$ induces an equivalence $C \to C^![-n]$. To consider non-weak Calabi-Yau structures we restrict to maps factoring through $\HCcx_*(A) \simeq \coHCcx_*(C)$.
\end{proof}

\section{Examples}
\subsection{Lie algebras}\label{sect:lie}
We turn to examples and begin with the observation that, while it is a priori difficult to check if a pointed curved coalgebra $C$ has a proper Calabi-Yau structure, the situation simplifies when $C$ is cocommutative.
\begin{definition}
    Let $C$ be a cocommutative coalgebra. We write $C^{(3)}$ for $C$ considered as a right $C$-comodule and left $C^e$-comodule by iterating the usual comultiplication.
\end{definition}
\begin{remark}
    Note that while any coalgebra is a bicomodule over itself this ``tricomodule'' structure only exists if $C$ is cocommutative, as the different coactions are not compatible otherwise.
\end{remark}
\begin{lemma}\label{lem:trimodule}
    Let $M \to N$ be weak equivalence of left $C$-comodules. Then there is an equivalence of $C$-bicomodules $M \boxempty_C C^{(3)} \to N \boxempty_C C^{(3)}$.
\end{lemma}
\begin{proof}
    This immediately follows by interpreting left $C^e$-comodules as $C$-bicomodules 
We have to show that $-\boxempty_C C^{(3)}$ preserves weak equivalences. 
As $C$ is an injective right comodule $-\boxempty_C C^{(3)}$ is exact (for such background facts on comodules we refer to \cite{doi1981homological}).
Thus $-\boxempty_C C^{(3)}$ preserves exact triples. 
Moreover as kernels and tensor products (over $k$) preserve direct sums so does the cotensor over $C$. 
It follows that $-\boxempty_C C^{(3)}$ preserves coacyclic objects and thus weak equivalences in the coderived category.
\end{proof}

\begin{corollary}\label{cor:cocommutative}
    Let $C$ be a finite-dimensional cocommutative pointed curved coalgebra with zero curvature. 
    Then  a cycle $\beta \in \coHNcx_n(C)$ is a proper $n$-Calabi-Yau structure if the induced map $C^* \to C[n]$ is a weak equivalence of $C$-comodules.
\end{corollary}
\begin{proof}
As $C$ has zero curvature $C$ and $C^*$ are naturally $C$-modules and one sees that $\beta$ induces a map $C^* \to C$ in $\Dco(C)$, which becomes $\phi(\beta)$  after cotensoring with $C^{(3)}$. We apply Lemma \ref{lem:trimodule} to  deduce that $C^* \to C[n]$ is also a weak equivalence of bicomodules.
\end{proof}

The corollary is very useful as it is easier to check weak equivalences for comodules than for bicomodules as we will see in our first example.

Let now $\mathfrak g$ be a finite dimensional Lie algebra of dimension $d$ over a field $k$. Then Chevalley-Eilenberg chains $C_*(\mathfrak g) \coloneqq C_*(\mathfrak g, k)$ form a cocommutative conilpotent dg coalgebra.
This is Koszul dual to the universal envelopping algebra $\cU(\mathfrak g)$ (see \cite{baranovsky2008universal}, this goes back to \cite{quillen1969rational}).

For a finite-dimensional Lie algebra $\mathfrak g$ we consider comodules over the Chevalley-Eilenberg chains $C_*(\mathfrak g)$. 
By Theorem \ref{thm:modulekoszul} we have $\De(\mathcal U(\mathfrak g)) \simeq \Dco(C_*(\mathfrak g))$.
As $\mathfrak g$-modules are $\mathcal U(\mathfrak g)$-modules, their derived category corresponds to  the coderived category of comodules over $C_*(\mathfrak g)$.

A Lie algebra $\mathfrak g$ is \emph{unimodular}
if $tr(\operatorname{ad}_{\mathfrak g} x) = 0$ for all $x \in \mathfrak  g$ or equivalently if $\bigwedge^d \mathfrak g^*$ is trivial.

\begin{remark}
    If $\mathfrak g$ is nilpotent or semisimple it is unimodular.
\end{remark}

\begin{lemma}\label{lem:ceadjunction}
    Write $A = \Omega C_*(\mathfrak g)$. We have  
    \[F(C_*(\mathfrak g)) = A\otimes^\tau C_*(\mathfrak g) \cong C_*(\mathfrak g, A)\]
    and 
    \[F(C^*(\mathfrak g)) = A \otimes^\tau C^*(\mathfrak g) \cong C^*(\mathfrak g, A),\] where the right hand sides are Chevalley-Eilenberg chains, resp.\ cochains, with coefficients in $A$ viewed as a $\mathfrak g$-module.
\end{lemma}
\begin{proof}
This follows from unravelling the definitions,
the twist of the tensor product corresponds exactly to the action of $\mathfrak g$ on $\mathcal U(\mathfrak g)$.
We use finite-dimensionality of $\mathfrak g$ for the second statement. 
As $C_*(\mathfrak g)$ is conilpotent we need not consider cotensor products here.
\end{proof}

We now recall Poincar\'e duality for Lie algebras.
\begin{theorem}\label{thm:liepd}
    Let $\mathfrak g$ be a $d$-dimensional Lie algebra and let $V$ be a dg $\mathfrak g$-module.
    Then there is a canonical quasi-isomorphism $C^*(\mathfrak g, V) \simeq C_*(\mathfrak g, V \otimes (\bigwedge^d \mathfrak g)^*)[d]$.
\end{theorem}
\begin{proof}
If $V$ is concentrated in degree 0
    this is essentially the main theorem of 
    \cite{hazewinkel1970duality}, the computation is also presented in
    \cite[Theorem 6.10]{knapp1988lie}. 
    The proof gives the desired quasi-isomorphism.
 If $V$ is a complex of $\mathcal U(\mathfrak g)$-modules the result follows by considering the canonical filtration on $V$ and observing that the associated spectral sequence is an isomorphism on the $E_1$-term.

    To define the map explicitly we specify $\phi: C^*(\mathfrak g) \to C_*(\mathfrak g)[-d]$ by choosing a generator $\theta$ of $\bigwedge^d \mathfrak g$ and sending $f \in C^k(\mathfrak g)$ to $(f\otimes \id)\Delta(\theta)$. 
    This is inverse to the map in \cite{knapp1988lie}.
\end{proof}

Putting this together we obtain the following. 
\begin{proposition}\label{prop:liepd}
Let $\mathfrak g$ be unimodular. Then there is a weak equivalence of comodules $C^*(\mathfrak g) \to C_*(\mathfrak g)[-d]$. 
Thus $C_*(\mathfrak g)$ is right $d$-Calabi-Yau.
\end{proposition}
\begin{proof}
    The map $\phi$ from the proof of Theorem \ref{thm:liepd} is a left comodule map (where $C^*$ has the comodule structure from \ref{lem:dualcomodule}) and does induce the Poincar\'e duality map for Lie algebras.

    We now set $V = \Omega C_*(\mathfrak g)$, considered as a differential graded $\mathfrak g$-module, in Theorem \ref{thm:liepd}
    and deduce from Lemma \ref{lem:ceadjunction} that $\phi$ is a weak 
 equivalence of $C$-comodules.

    It then follows from Corollary \ref{cor:cocommutative} that there is a a weak equivalence of bicomodules  $C^*(\mathfrak g) \to C_*(\mathfrak g)[-d]$.
    Thus Lemma \ref{lem:computeadjoint} and Theorem \ref{thm:main} imply the weak version of the second part of the statement.

  It remains to show that the map is a cyclic cycle. 
  We translate the map $f \mapsto (f\otimes \id)\Delta(\theta)$ from Theorem \ref{thm:liepd} to an element of $C_*(\mathfrak g)\otimes^\pi \Omega C_*(\mathfrak g)$ 
  which  takes the form $\theta^{(1)}[\theta^{(2)}]$ in Sweedler notation.
  Then $B(\theta^{(1)}[\theta^{(2)}]) = [\theta^{(1)}|\theta^{(2)}] - [\theta^{(2)}|\theta^{(1)}] = 0$ by cocommutativity of $C$.
\end{proof}

Thus we obtain: 
\begin{theorem}\label{thm:liekoszul}
    Let $\mathfrak g$ be a $d$-dimensional Lie algebra. Then the following are equivalent:
    \begin{enumerate}
        \item $\mathcal U(\mathfrak g)$ has a smooth $d$-Calabi-Yau structure 
        \item $C_*(\mathfrak g)$ has a proper $d$-Calabi-Yau structure.
         \item $\mathfrak  g$ is unimodular.
    \end{enumerate}
\end{theorem}
\begin{proof}
The equivalence of (1) and (2) is Theorem \ref{thm:main} since $\Omega C_*(\mathfrak g) \simeq \mathcal U(\mathfrak g)$.

(3) implies (2) by Proposition \ref{prop:liepd}. 

We consider the converse. 
As $\Omega C_*(\mathfrak g)^e$ is the universal 
$\mathfrak g$-bimodule  the equivalence $C^*(\mathfrak g, \Omega(C_*\mathfrak g)^e) \to C_*(\mathfrak g, \Omega(C_*\mathfrak g)^e[d])$ from the proper Calabi-Yau structure also implies Poincar\'e duality for coefficients in the trivial bimodule $k$ and thus $H^0(\mathfrak g, (\bigwedge^d \mathfrak g)^*) \simeq H_n(\mathfrak g, (\bigwedge^d \mathfrak g)^*)$. Meanwhile Theorem \ref{thm:liepd} implies $H^0(\mathfrak g, k) \cong H_n(\mathfrak g, (\bigwedge^d \mathfrak g)^*) $. 
As $H^0$ is just the maximal trivial submodule these statements together imply $\bigwedge^d \mathfrak g$ is trivial and $\mathfrak g$ is unimodular.   
\end{proof}
The weak Calabi-Yau structure on the universal enveloping algebra of a unimodular Lie algebra can be found in \cite[Lemma 4.1]{he2010cocommutative}, relying on \cite{yekutieli2000rigid}.
As far as we are aware the lift of the equivalence to a class in cyclic homology is new.

\begin{remark}
We may consider the situation if $k$ is not a field.
Then we assume $\mathfrak g$ is a free $k$-module of finite rank and Theorem \ref{thm:liepd} still holds, see \cite{hazewinkel1970duality}.
Proposition \ref{prop:liepd} also holds and we may deduce the existence of a smooth $n$-Calabi-Yau structure on $\Omega C_*(\mathfrak g)$.
\end{remark}

\subsection{Poincaré duality}
\label{sect:manifold}
In this section we work over an arbitrary PID $k$. Let $K$ be a simplicial complex. Denote by $K^{\dagger}$ the simplicial set obtained by applying the nerve functor to the poset of face inclusions of $K$ and by $C^{\Delta}_*(K;k)$ the dg coalgebra of normalized simplicial chains on $K^{\dagger}$. Denote by $C_*(K;k)$ be the dg coalgebra of normalized \textit{singular} chains on the geometric realization of $K$. The coproduct of both $C^{\Delta}_*(K;k)$ and $C_*(K;k)$ is given by the Alexander-Whitney diagonal approximation. Recall there is a natural quasi-isomorphism of coalgebras $j \colon C^{\Delta}_*(K;k) \to C_*(K;k)$ that consideres a simplex in the barycentric subdivision of $K$ as a singular simplex in $|K|$. The natural map $j$ also induces a quasi-isomorphism with respect to coefficients in any local system. 

If $b \in K_0$ is a fixed vertex, we denote by $C_*(K,b;k)$ the reduced singular chains defined by first considering the chain complex generated by singular simplices $\sigma \colon \Delta^n \to |K|$ sending all vertices to $b$ and then quotienting out the sub-complex of degenerate chains. The Alexander-Whitney coproduct induces a conilpotent dg coalgebra structure on $C_*(K,b;k)$ and, if $K$ is connected, the natural map $i \colon C_*(K,b;k) \to C_*(K,k)$ is a quasi-isomorphism. In fact, any collection \begin{eqnarray} \label{choiceofpaths} \mathcal{O}=\{ \gamma_x \colon [0,1] \to |K| : x \in |K|, \gamma(0)=x, \gamma(1)=b\}
\end{eqnarray} of paths from each point in $|K|$ to $b$ gives rise to a map \[g_{\mathcal{O}} \colon C_*(K,k) \to C_*(K,b;k)\] chain homotopy inverse to $i$. For any singular simplex $\sigma \colon \Delta^n \to |K|$, \[g_{\mathcal{O}} (\sigma) \colon \Delta^n \to |K|\] is defined by using the paths in $\mathcal{O}$ to connect all vertices of $\sigma$ to $b$. 

 \begin{proposition} \label{prop: proper}
If $(K,b)$ is a connected pointed finite simplicial complex, then $C_*(K,b;k)$ is a proper conilpotent dg coalgebra. 
\end{proposition}
 
\begin{proof}
We prove $C_*(K,b;k)$ is weakly equivalent as a $C_*(K,b;k)$-bicomodule to a $C_*(K,b;k)$-bicomodule whose underlying graded $k$-module is finitely presented and we will be done by Lemma \ref{lem:computeadjoint}. Since $K$ has a finite number of simplices, $C^{\Delta}_*(K;k)$ is finitely presented as a graded $k$-module. Since $K$ is connected, any choice $\mathcal{O}$ as in \ref{choiceofpaths}, gives rise to a quasi-isomorphism of dg coalgebras 
 \[ \rho_{\mathcal{O}} \colon C^{\Delta}_*(K;k) \xrightarrow{j} C_*(K;k) \xrightarrow{g_{\mathcal{O}}} C_*(K,b;k),\]
 where $j$ and $g_{\mathcal{O}}$ are as defined above. For simplicity, write $C= C_*(K,b;k)$ when considered as a dg coalgebra. We equip $C^{\Delta}_*(K;k)$ with a $C$-bicomodule structure using its dg coalgebra structure followed by $\rho_{\mathcal{O}}$. We now observe the induced map 
 \[ F_{C}(\rho_{\mathcal{O}}) \colon  C^{\Delta}_*(K;k) \otimes^{\tau^e} (\Omega C)^{e} \to C_*(K,b;k) \otimes^{\tau^e} (\Omega C)^{e} \]
 is a quasi-isomorphism, which, by definition, means $\rho_{\mathcal{O}}$ is a weak equivalence of $C$-bicomodules. In fact,  the above map preserves the convergent filtrations given by $\mathcal{F}_p =  \big( \bigoplus_{q \leq p} C^{\Delta}_q(K;k)\big) \otimes^{\tau^e} (\Omega C)^{e} $ and $\mathcal{F}'_p=  \big(\bigoplus_{q \leq p} C_q(K,b;k)\big) \otimes^{\tau^e} (\Omega C)^{e} $. On the $E_2$-page of the corresponding spectral sequences, the map induced by 
$F_{C}(\rho_{\mathcal{O}})$ is precisely the classical isomorphism between the simplicial homology of $K$ and the singular homology of $|K|$ both with coefficients in the local system $H_*((\Omega C)^{e})$. 
\end{proof}

We say $K$ is an \textit{oriented Poincaré duality complex of formal dimension $n$} if $K$ is a connected finite simplicial complex equipped with a cycle $\alpha_K \in C_n(K;k)$ such that, for any vertex $b$ in $K$ and any left $k[\pi_1(K,b)]$-module $\ell$,
the cap product with $\alpha_K$ induces a degree $n$ quasi-isomorphism
\[ C^*_{\Delta}(K;\ell) \xrightarrow{\simeq} C_{n-*}^{\Delta}(K;\ell)\]
from simplicial cochains to chains with coefficients in $\ell$ (using the notation introduced in the proof of Proposition \ref{prop: proper}.) Triangulated oriented closed manifolds are examples of oriented Poincaré duality complexes, see \cite[Theorem 10.2]{spanier1993singular} in the case $X = A$ is closed and $B = \emptyset$. We call $[\alpha_K] \in H_n(K;k)$ the fundamental class of $K$. The chain complexes $C^*_{\Delta}(K;\ell)$ and $C_{*-n}^{\Delta}(K;\ell)$ are given by the twisted tensor products $C^*_{\Delta}(K;k) \otimes^{p_0 \circ \tau} \ell$ and $C_*^{\Delta}(K;k) \otimes^{p_0 \circ \tau} \ell$, respectively, where $\tau$ is the twisting cochain given by the composition \[\tau \colon C_*^{\Delta}(K;k) \xrightarrow{j}  C_*(K;k) \xrightarrow{g_{\mathcal{O}}} C_*(K,b;k) \to \Omega C_*(K,b;k),\] the last map being the universal twisting cochain, and \[p_0 \colon \Omega C_*(K,b;k) \to H_0( \Omega C_*(K,b;k)) \cong k[\pi_1(K,b)]\] is the canonical projection. 

We now take into account a natural dg bialgebra structure on $\Omega C_{*}(K,b;k)$ to lift Poincaré duality to a chain level homotopy coherent map of dg bicomodules over the dg coalgebra singular chains.
Recall that if $C$ is the dg coalgebra of normalized simplicial chains on any simplicial set $S$ then there is a natural co-action of the surjection operad (a particular model for the dg $E_{\infty}$-operad) extending the dg coassociative coproduct of $C$. If $S$ is reduced (i.e. $S_0$ is a singleton), the $E_2$ part of this structure induces a coproduct 
\[\nabla \colon \Omega C \to \Omega C \otimes \Omega C\]
making $\Omega C$ into a dg bialgebra. Another way of describing this coproduct structure is by identifying $\Omega C$ with the normalized cubical chains on certain cubical set and using the cubical diagonal approximation. In the many object case, i.e. when $S$ has an arbitrary number of vertices, this structure may be interpreted as dg coalgebra enrichment on the category $\Omega C$. See \cite[Section 4.2]{rivera2022algebraic} for more details regarding this structure. 

In general, the dg bialgebra $\Omega C$ might not have the property of being a dg Hopf algebra, i.e. an antipode might not exist. However, we may formally invert all $1$-simplices $\sigma \in S_1$ in $\Omega C$, which satisfy $\nabla [\sigma ] = [\sigma] \otimes [\sigma]$, to obtain a new dg bialgebra, denoted by $\widehat \Omega C$, that has the property of being a dg Hopf algebra and the square of the antipode is chain homotopic to the identity map  \cite[Theorem 23]{rivera2022algebraic}. Furthermore, it turns out that $\widehat \Omega C$ is quasi-isomorphic, as a dg bialgebra, to the singular chains on the space of based (Moore) loops in $|S|$. If $S$ is a pointed \textit{Kan complex}, then the natural map of dg bialgebras $\Omega C \to \widehat{\Omega}C$ is a quasi-isomorphism. We recall the following classical observation \cite{bichon, malm, rivera2022algebraic}. 

\begin{lemma}\label{lem:hopf}
  Let $A$ be a dg Hopf algebra with coproduct $\nabla \colon A \to A \otimes A$, counit  $\varepsilon \colon A \to k$, and antipode $s \colon A \to A$. For any dg $A$-bimodule $M$, there are isomorphisms in the derived category of $k$-chain complexes
\begin{eqnarray}
A \otimes^L_{A^e} M \simeq k \otimes_A^L M_{\text{ad}}
\end{eqnarray} 
and, if $s^2=s \circ s \colon A \to A$ is chain homotopic to the identity map,  
\begin{eqnarray}
   \RHom_{A^e}(A,M) \simeq  \RHom_A(k,M_{\text{ad}}),
\end{eqnarray}
where $k$ is considered as a left $A$-module via the counit $\varepsilon \colon A \to k$ and $M_{\text{ad}}$ is $M$ considered as a left $A$-module via the action \[a \cdot m = \sum (-1)^{|a'||a''|} a''ms(a'),\]
writing $\nabla(a)=\sum a' \otimes a''$.   
\end{lemma}

We are now ready to prove the following extension of classical Poincaré duality.

\begin{theorem}\label{thm:poincare}
    Let $(K,b)$ be a pointed oriented Poincaré duality complex and $\alpha_K \in C^{\Delta}_n(K;k)$ a representative of the fundamental class.  The map \[- \frown \alpha_K \colon C^*_{\Delta}(K;k) \xrightarrow{\simeq} C_{n-*}^{\Delta}(K;k)\]
    is a map of right comodules over the dg $k$-coalgebra $C_*(K,b;k)$ and extends to a degree $n$ quasi-isomorphism
    \[ P_K \colon C^*_{\Delta}(K;k) \otimes^{\tau^e} (\Omega C_*(K,b;k) )^e  \xrightarrow{\simeq}  C_*^{\Delta}(K;k) \otimes^{\tau^e} (\Omega C_*(K,b;k) )^e .\]
\end{theorem}

\begin{proof}
For simplicity write $C=C_*(K,b;k)$, $E=C_*^{\Delta}(K;k)$, and $E^*= C^*_{\Delta}(K;k)$ and consider the latter two as bicomodules over the coaugmented dg coalgebra $C$. The map induced by capping with a chain on the right is clearly a strict map of right $C$-comodules; however, it does not strictly preserve the left $C$-comodule structure. By Lemma \ref{lem:locsytems}, which will be proved below, the local coefficient version of Poincaré duality implies that for any non-negatively graded left dg $\widehat{\Omega} C$-module $L$ we have a degree $n$ quasi-isomorphism of dg $k$-modules
\[ - \frown \alpha_K \colon E^* \otimes^{\tau} L \xrightarrow{\simeq} E \otimes^{\tau} L.\]
The twisted tensor product in the domain of the above map is a model for the derived hom space $\RHom_{\widehat\Omega C}(k,L)$ and the range is a model for the derived tensor product $k \otimes^{L}_{\widehat \Omega C} L$. In fact, we have that $\widehat \Omega C \otimes^{\tau} E$ (and $E \otimes^{\tau}  \widehat \Omega C$) is a free resolution of $k$ as a left (right) dg $\widehat\Omega C$-module and there are natural isomorphisms of dg $k$-modules
\[\uHom_{\widehat \Omega C}(\widehat \Omega C \otimes^{\tau} E, L) \cong E^* \otimes^{\tau} L \]
and
\[ (  E \otimes^{\tau}  \widehat \Omega C) \otimes_{\widehat\Omega C}\widehat\Omega C \cong E \otimes^{\tau} L.\] 
Applying the above to $L=(\widehat\Omega C \otimes \widehat\Omega C)_{\text{ad}}$
we obtain a degree $n$ quasi-isomorphism
\[  \uHom_{\widehat \Omega C}(\widehat \Omega C \otimes^{\tau} E, (\widehat\Omega C \otimes \widehat\Omega C)_{\text{ad}} )\xrightarrow{\simeq} E \otimes^{\tau} (\widehat\Omega C \otimes \widehat\Omega C)_{\text{ad}}\]
induced by capping with $\alpha_M$. Finally, we obtain the desired quasi-isomorphism by noting that Lemma \ref{lem:hopf} provides natural quasi-isomorphisms
\[ \uHom_{\widehat \Omega C}(\widehat \Omega C \otimes^{\tau} E, (\widehat\Omega C \otimes \widehat\Omega C)_{\text{ad}} )\simeq \uHom_{(\widehat \Omega C)^e}(E \otimes^{\tau^e} (\Omega C)^e, (\widehat \Omega C)^e)  \cong 
E^* \otimes^{\tau^e} (\Omega C)^e\]
and
\[ E \otimes^{\tau} (\widehat\Omega C \otimes \widehat\Omega C)_{\text{ad}} 
\cong E \otimes^{\tau^e} (\Omega C)^e.
\qedhere
\]
\end{proof}
\begin{lemma}\label{lem:locsytems} Let $\tau \colon C \to A$ be a twisting cochain from a dg coalgebra to a non-negatively graded dg algebra and note by $p_0 \colon A \to H_0(A)$ the canonical projection. If a map $f \colon E \to E'$ of non-negatively graded right dg $C$-comodules induces a quasi-isomorphism \[f \otimes \text{id}_{\ell} \colon \colon E \otimes^{p_0 \circ \tau} \ell \xrightarrow{\simeq} E' \otimes^{p_0 \circ \tau} \ell \] for all left $H_0(A)$-modules $\ell$, then it induces a quasi-isomorphism \[f \otimes \text{id}_L \colon E \otimes^{\tau} L \xrightarrow{\simeq} E' \otimes^{\tau} L\] for all dg $A$-modules $L$.
\end{lemma}
 \begin{proof}   For any right dg $C$-comodule $E$ and any left dg $A$-module $L$, the twisted tensor product $E \otimes^{\tau} L$ has a canonical filtration given by 
 $\mathcal{F}_p= \bigoplus_{q \leq p} E_q \otimes L$. The differential $\partial \colon E \otimes^{\tau} L \to E \otimes^{\tau} L$ decomposes as a direct sum of three maps
 $\partial_E \otimes \text{id}_L \colon \mathcal{F}_p \to \mathcal{F}_{p-1}$,
 $\text{id}_{E} \otimes \partial_L \colon \mathcal{F}_p \to F_{p}$, and $\partial_{\tau} \colon \mathcal{F}_p \to \mathcal{F}_{p-1}$, where the latter uses the $C$-comodule structure of $E$, the twisting cochain $\tau$, and the $A$-module structure of $L$. Furthermore, $\partial_{\tau} \colon \mathcal{F}_p \to \mathcal{F}_{p-1}$ decomposes as a direct sum of maps
 $\partial_{\tau}^{1} \colon \mathcal{F}_p \to \mathcal{F}_{p-1}$ and $\partial_{\tau}^{2} \colon \mathcal{F}_p \to \mathcal{F}_{p-2}$, where $\partial_{\tau}^{1}$ corresponds to the summand $E_p \to E_{p-1} \otimes C_1$ in the $C$-coaction of $E$. 
 By examining the induced differential, we obtain that the $E_1$-page of the corresponding spectral sequence is a twisted tensor product of the form $C \otimes^{p_0 \circ \tau} H_*(L)$. Hence, if $f \colon E \to E'$ induces a quasi-isomorphism with respect to coefficients in any left $H_0(A)$-module, it induces a quasi-isomorphism on the $E_1$-page of the spectral sequence associated to the twisted tensor product with coefficients in any non-negatively graded left dg $A$-module.
 Since the filtration $\mathcal F$ is complete and exhaustive this implies that $f$ induces a quasi-isomorphism $E \otimes^{\tau} L \xrightarrow{\simeq} E' \otimes^{\tau} L$ for any left dg $A$-module $L$. 
\end{proof}

Coming back to Theorem \ref{thm:poincare}, denote $C=C_*(M;k)$ and $E=C^{\Delta}_*(M;k)$. Recall the coHochschild cochain complex is defined as 
\[ co\mathcal{CH}^*(E^*,C)=\uHom_{(\Omega C)^{e}}( E^* \otimes^{\tau^e} (\Omega C)^e, \Omega C ).\]
The natural quasi-isomorphism of $\Omega C$-bimodules $E \otimes^{\tau^e} (\Omega C)^e \xrightarrow{\simeq} \Omega C$ from Lemma \ref{lem:aec} gives rise to a quasi-isomorphism
\begin{eqnarray}\label{coHH^*} \Phi \colon  \uHom_{(\Omega C)^{e}}(E^* \otimes^{\tau^e} (\Omega C)^e, E \otimes^{\tau^e} (\Omega C)^e) \xrightarrow{\simeq} \coHHcx^*(E^*,C).
\end{eqnarray}
The next result uses the fact that, for any connected pointed simplicial complex $(K,b)$, there are natural isomorphisms
\[\coHH_*(C_*(K,b;k)) \cong H_*(L|K|;k)\]
and
\[ \coHN_*(C_*(K,b;k)) \cong H^{S^1}_*( L|K|;k),\]
where $L|K|$ is the free loop space on $K$, and $H^{S^1}_*(L|K|)$ denotes the $S^1$-equivariant homology of $L|K|$ equipped with the $S^1$-action given by rotating loops
\cite{rivera2024cyclichomologycategoricalcoalgebras, rivera2022algebraic}.

\begin{proposition} \label{prop:mfdpropercy}
Let $(K,b)$ as in Theorem \ref{thm:poincare}. Denote $C=C_*(K,b;k)$ and $E=C^{\Delta}_*(K;k)$. Then the chain map
\[ P_K \colon E^* \otimes^{\tau^e} (\Omega C)^e \xrightarrow{\simeq}  E \otimes^{\tau^e} (\Omega C)^e\]
of Theorem \ref{thm:poincare} determines a proper Calabi-Yau structure on $C$.
\end{proposition}

\begin{proof}
By Lemma \ref{lem:computeadjoint} $E^* \simeq C^\vee$ as a $C$-bicomodule. Thus by Lemma \ref{lem:hochschildmap} the map $P_K$ determines a class in $\coHH_*(C)$ defining a weak proper Calabi-Yau structure on $C$.  This class coincides with the image of the fundamental class of $K$  under the composition of maps
\[ H_*(|K|;k) \rightarrow H_*(L|K|;k) \cong \coHH_*(C). \]
where the first map is induced by the constant loops inclusion map $|K| \hookrightarrow L|K|$. Since homology classes of constant loops lift  to $S^1$-equivariant homology, and $H_*^{S^1}(L|K|;k) \cong \coHN_*(C)$ \cite{rivera2024cyclichomologycategoricalcoalgebras},
this weak Calabi-Yau structure lifts to a Calabi-Yau structure. 
\end{proof}

\begin{corollary} \label{cor:CYonK}
Let $(K,b)$ be as in Theorem \ref{thm:poincare}. The fundamental class of $K$ gives rise to a proper Calabi-Yau structure on the coaugmented dg coalgebra of singular chains on $K$ and, consequently, a smooth Calabi-Yau structure on the dg algebra of singular chains on the space $\Omega_b K$ of (Moore) loops in $|K|$ based at $b$. 
\end{corollary}
\begin{proof}
By applying Theorem \ref{thm:main} to Proposition \ref{prop:mfdpropercy} we obtain a smooth Calabi-Yau structure on $\Omega C_*(K,b;k)$. Finally, we use that there is a natural quasi-isomorphism of dg algebras from $\Omega C_*(K,b;k)$ to the singular chains on $\Omega_bK$ \cite{rivera-zeinalian, riveracobar}.
\end{proof}
Finally, we put together the above results to prove Theorem \ref{thm:poincareintro}.
\begin{proof}[Proof of Theorem \ref{thm:poincareintro}] 

Suppose $K$ is a finite simplicial complex homotopy equivalent to $X$ and let $\alpha_K \in C^{\Delta}_n(K;k)$ represent the class $[\alpha_X] \in H_n(X;k)\cong H_n(K;k)$. First note that Theorem \ref{thm:main} implies (2) and (3) are equivalent. By Corollary \ref{cor:CYonK}, (1) implies (2). We show that (2) implies (1). If (2) is satisfied, the cycle $\alpha_K$ gives rise to  a quasi-isomorphism
\[ C^*_{\Delta}(K;k) \otimes^{\tau^e} (\Omega C_*(K,b;k) )^e  \xrightarrow{\simeq}  C_*^{\Delta}(K;k) \otimes^{\tau^e} (\Omega C_*(K,b;k) )^e.\] Since $\tau^e \colon C^*_{\Delta}(K;k)^e \to (\Omega C_*(K,b;k) )^e$ is the universal twisting cochain, we have that, for any left dg $(\Omega C_*(K,b;k) )^e$-module $L$, the cycle $\alpha_K$ gives rise to a quasi-isomorphism
\[ C^*_{\Delta}(K;k) \otimes^{\tau^e}  L \xrightarrow{\simeq}  C_*^{\Delta}(K;k) \otimes^{\tau^e} L. \]
In particular, any left $\pi_1(K,b)$-module $\ell$ may be regarded as a dg left $(\Omega C_*(K,b;k) )^e$-module $\ell^e$ with trivial differential and with action $(a \otimes b)\cdot x= p_0(a)x\varepsilon(b)$, where $a \otimes b \in (\Omega C_*(K,b;k) )^e$, $x \in \ell$, $p_0 \colon \Omega C_*(K,b;k) \to H_0(\Omega C_*(K,b;k))\cong \pi_1(K,b)$ is the natural projection map, and $\varepsilon \colon \Omega C_*(K,b;k) \to k$ the augmentation. We have isomorphisms of chain complexes
\[C^*_{\Delta}(K;k) \otimes^{\tau^e}  \ell^e \cong C^*_{\Delta}(K;k) \otimes^{\tau} \ell\]
and
\[ C_*^{\Delta}(K;k) \otimes^{\tau^e}  \ell^e \cong C_*^{\Delta}(K;k) \otimes^{\tau} \ell,\]
since the composition $C \xrightarrow{\tau} \Omega C \xrightarrow{\varepsilon} k$ is the zero map. Thus $\alpha_K$ gives rise to a quasi-isomorphism
\[C^*_{\Delta}(K;k) \otimes^{\tau} \ell \xrightarrow{\simeq} C_*^{\Delta}(K;k) \otimes^{\tau} \ell,\]
for any left $\pi_1(K,b)$-module $\ell$, as desired. 
\end{proof}

\printbibliography

\end{document}